\documentclass[10pt, article]{amsart}
\usepackage{tikz}
\usetikzlibrary{calc}
\usepackage{ae} 
\usepackage[T1]{fontenc}
\usepackage[cp1250]{inputenc}
\usepackage{amsmath}
\usepackage{amssymb, amsfonts,amscd,verbatim}
\usepackage{faktor}
\usepackage[normalem]{ulem}
\usepackage{hyperref}
\usepackage{indentfirst}
\usepackage{latexsym}
\input xy
\xyoption{all}

\usepackage{xcolor}

\usepackage{amsmath}    

\theoremstyle{plain}
\newtheorem{Pocz}{Poczatek}[section]
\newtheorem{Proposition}[Pocz]{Proposition}

\newtheorem{Theorem}[Pocz]{Theorem}
\newtheorem{Corollary}[Pocz]{Corollary}

\newtheorem{Lemma}[Pocz]{Lemma}

\newtheorem{Notation}[Pocz]{Notation}

\newtheorem{Example}[Pocz]{Example}

\theoremstyle{definition}
\newtheorem{Definition}[Pocz]{Definition}

\theoremstyle{remark}

\numberwithin{equation}{section}
\title[The Stone's Representation Theorems and compactifications of a discrete space]
{The Stone's Representation Theorems and compactifications of a discrete space}

\author{Hussain Rashed}
\address{University of Tennessee, Knoxville, TN 37996, USA}
\email{hrashed1@vols.utk.edu}

\date{ \today
}


\begin{document}
\maketitle
\begin{center}
\today
\end{center}

\tableofcontents

\begin{abstract}
This paper is meant to give a short exposition of the Stone's Representation Theorems. We provide three equivalent approaches to construct a Stone's space from a given Boolean algebra. Finally, we utilize the Stone's Representation Theorems to study the compactifications of a discrete space.
\end{abstract}

\section{Introduction}
 The work by Marshall Stone \cite{Stone} associated to every Boolean algebra $B$ a totally disconnected compact Hausdorff space $X_B$ called the Stone's space from which $B$ can be restored. Conversely, every Stone's space $X$ can be reconstructed from an associated Boolean algebra $B_X$. These two facts are the core of the Stone's Representation Theorem for Boolean algebras, and the Stone's Representation Theorem for Stone's spaces, respectively. The main aim of this self-contained exposition is to provide detailed proofs of the Stone's Representation Theorems using modern tools. The exposition can be outlined in the following points:\\
 $\bullet$ The family of all clopen subsets of a topological space $X$ forms a Boolean algebra $Clop(X)$ \ref{ClopBoolean}. If $B$ is a Boolean algebra and $X_B$ is its assigned Stone's space, then $B$ is the clopen Boolean algebra of $X_B$ \ref{StoneRepThOfBooleanAlg}. Furthermore, a homomorphism $\varphi:B_1\to B_2$ between Boolean algebras induces a continuous map $\widetilde{\varphi}:X_{B_2}\to X_{B_1}$ between their associated Stone's spaces; if $\varphi:B_1\to B_2$ is an isomorphism, then $\widetilde{\varphi}:X_{B_2}\to X_{B_1}$ is a homeomorphism \ref{IsoBooleansHaveHomeoStonesSpaces}. The Stone's space $X_B$ associated to a given Boolean algebra can be obtained in three equivalent ways:\\
($1$) As the space of all ultrafilters of $B$ and denoted by $\mathcal{U}(B)$.\\
($2$) As the space of all homomorphisms from $B$ to $\mathbb{Z}_2$ and denoted by $Hom(B,\mathbb{Z}_2)$.\\
($3$) As the spectrum of its induced Boolean ring $R_B$ and denoted by $Spec(R_B)$.\\
The Stone's spaces $\mathcal{U}(B)$, $Hom(B,\mathbb{Z}_2)$ and $Spec(R_B)$ are homeomorphic \ref{StoneAsSetOfHom}, \ref{StoneAsSpec}.\\
$\bullet$ A continuous map $f:X\to Y$ between topological spaces induces a homomorphism $\widetilde{f}:Clop(Y)\to Clop(X)$ between their clopen Boolean algebras. If $f:X\to Y$ is a homeomorphism, $\widetilde{f}:Clop(Y)\to Clop(X)$ is an isomorphism \ref{ClopBoolean}. Moreover, if $X$ is a Stone's space, then it is the Stone's space associated to the Boolean algebra $Clop(X)$ \ref{StoneRepThOfBooleanAlg}.\\
In a categorical language, the above can be re-stated as that there is a duality between the category of Boolean algebras and the category of Stone's spaces.\\ 
$\bullet$ When $X$ is a discrete space; a topological space $Y$ is a compactification of $X$ if and only if it is homeomorphic to a Stone's space $\mathcal{U}(\mathcal{B})$ for some a sub-Boolean algebra $\mathcal{B}$ of $\mathcal{P}(X)$ that separates points of $X$ \ref{ComptionOFDiscreteSp}; when $\mathcal{B}=\mathcal{P}(X)$, $\mathcal{U}(\mathcal{B})$ is the Stone-$\breve{C}$ech compactification of $X$. In particular, $Hom(B,\mathbb{Z}_2)$ and $Spec(R_B)$, where $B=\mathcal{P}(X)$, are distinct approaches to  construct the Stone-$\breve{C}$ech compactification of $X$ \ref{WaysToConsSTCompOfDiscrete}.\\

\section{Boolean algebras and their homomorphisms}
\begin{Definition}
A \textbf{Boolean algebra} $(B,+,\cdot,\neg)$ is a set $B$ equipped with two binary operations $+$, $\cdot $ and a unary operation $\neg$ called the complement, where the following are satisfied:\\
$(1)$ $+$ and $\cdot $ are \textbf{commutative} and \textbf{associative}, that is:\\
$\bullet$ $x+y=y+x$ and $x\cdot y=y\cdot x$, for all $x,y\in B$.\\
$\bullet$ $x+(y+z)=(x+y)+z$ and $x\cdot(y\cdot z)=(x\cdot y)\cdot z$, for all $x,y,z \in B$.\\
$(2)$ $+$ is \textbf{distributive over} $\cdot $, and $\cdot $ is distributive over $+$, that is:\\
$\bullet$ $x+(y\cdot z)=(x+y)\cdot(x+z)$, for all $x,y,z \in B$.\\
$\bullet$ $x\cdot(y+ z)=(x\cdot y)+(x\cdot z)$, for all $x,y,z \in B$.\\
$(3)$ $+$ and $\cdot $ are \textbf{connected by the absorption law}, that is:\\
$\bullet$ $x+(x\cdot y)=x\cdot(x+y)=x$, for all $x,y\in B$.\\
$(4)$ There exist two elements $0,1\in B$ satisfying:\\
$\bullet$ $x+(\neg x)=1$, for all $x\in B$;\\
$\bullet$ $x\cdot(\neg x)=0$, for all $x\in B$.
\end{Definition}
\begin{Example}\label{PowersetBooleanAlg}
Consider the power set $\mathcal{P}(X)$ of $X$; then $(\mathcal{P}(X),+,\cdot,\neg)$ is a Boolean algebra, where $+:=\cup$, $\cdot:=\cap$, $0:=\emptyset$, $1:=X$ and $\neg A:=X\setminus A$ for all $A\in \mathcal{P}(X)$. The Boolean algebra $(\mathcal{P}(X),\cup,\cap,\neg)$ is called the \textbf{power set Boolean algebra}.
\end{Example}
\begin{Definition}
Let $(B,+,\cdot,\neg)$ be a Boolean algebra, a subset $A\subset B$ is called a \textbf{sub-Boolean algebra} if $(A,+,\cdot,\neg)$ is itself a Boolean algebra.
\end{Definition}
\begin{Example}
Let $X$ be a topological space. The family of all clopen subsets $Clop(X):=\{U\subset X:U$ is clopen$\}$ is a sub-Boolean algebra of $\mathcal{P}(X)$. The Boolean algebra $(Clop(X),+,\cdot,\neg)$ is called the \textbf{clopen Boolean algebra} associated with $X$. 
\end{Example}
\begin{Example}
Let $G$ be a group acts on $X$; a subset $A\subset X$ is called \textbf{$G$-commensurated} or \textbf{$G$-almost invariant} if $g\cdot A\Delta A$ is a finite subset of $X$ for all $g\in G$. The family of all $G$-commensurated subsets of $X$ is denoted by $Comm_G(X)$. $Comm_G(X)$ is a sub-Boolean algebra of $\mathcal{P}(X)$ \cite{Corn}.
\end{Example}
\begin{Example}\label{SpaceOEnds}
Let $G$ be a group acts on $X$; define a relation $\sim$ on $Comm_G(X)$ by: for all $A,B\in Comm_G(X)$, $A\sim B$ if and only if $A\Delta B$ is finite. Then:\\
($a$) The relation $\sim$ is an equivalence relation on $Comm_G(X)$.\\
($b$) The family of all equivalence classes $\faktor{Comm_G(X)}{\sim}:=\{[A]:A\in Comm_G(X)\}$ of $Comm_G(X)$ is a Boolean algebra with operations:\\
$\bullet$ $[A]+[B]:=[A\cup B]$;\\
$\bullet$ $[A]\cdot[B]:=[A\cap B]$;\\
$\bullet$ $\neg[A]:=[X\setminus A]$.
\end{Example}

The fact that the binary relations $+$ and $\cdot $ being connected by the absorption law gives us the following properties.
\begin{Proposition}
Let $(B,+,\cdot,\neg)$ be a Boolean algebra, the following are satisfied:\\
$(a)$ $\cdot$ is \textbf{idempotent}, that is $x^2:=x\cdot x=x$, for all $x\in B$.\\
$(b)$ $+$ is \textbf{idempotent}, that is $2x:=x+x=x$ for all $x\in B$.\\
$(c)$ $1\cdot x=x$, for all $x\in B$.\\
$(d)$ $1+ x=1$, for all $x\in B$.\\
$(e)$ $0+x=x$, for all $x\in B$.\\
$(f)$ $0\cdot x=0$, for all $x\in B$.
\end{Proposition}
\begin{proof}
$(a)$ Let $y\in B$, then $x=x+(x\cdot y)$ and therefore $x^2=x\cdot (x+(x\cdot y))=x$.\\
$(b)$ By part $a)$ we have $2x=x^2+x^2=x\cdot (x+x)=x$.\\
$(c)$ $x\cdot 1=x\cdot (x+(\neg x))=x$.\\
$(d)$ $1+ x=(x+(-x))+x=2x+(\neg x)=x+(\neg x)=1$.\\
$(e)$ $0+x=x+0=x+(x\cdot (\neg x))=x$.\\
$(f)$ $0\cdot x=x\cdot 0=x\cdot (x\cdot(\neg x))=x^2\cdot(\neg x)=x\cdot(\neg x)=0$.
\end{proof}
\begin{Proposition}\label{ComplementProperties}
Let $(B,+,\cdot,\neg)$ be a Boolean algebra and $x,y\in B$. Then:\\
$(a)$ If $x\cdot y=0$ and $x+y=1$, then $y=\neg x$.\\
$(b)$ $\neg(\neg x)=x$.\\
$(c)$ $\neg 1=0$ and $\neg 0=1$.
\end{Proposition}
\begin{proof}
Notice that part $(b)$ and $(c)$ follow from $(a)$. Let $x,y\in B$, such that $x\cdot y=0$ and $x+y=1$. On one hand, one has $\neg x=\neg x\cdot 1=\neg x\cdot (x+y)=(\neg x)\cdot x+(\neg x)\cdot y=0+(\neg x)\cdot y=\neg x\cdot y$. On the other hand, $\neg x=\neg x+0=(\neg x\cdot y)+(x\cdot y)=(\neg x+x)\cdot y=1\cdot y=y$. 
\end{proof}
\begin{Proposition}
Let $(B,+,\cdot,\neg)$ be a Boolean algebra, and define a relation $\leqslant$ on $B$ by: $x\leqslant y$ if and only if $x+y=y$ for all $x,y\in B$. Then the relation $\leqslant$ is a partial order relation on $B$.
\end{Proposition}
\begin{proof}
$\bullet$ Reflexivity: for all $x\in B$, $x+x=x^2+x^2=x\cdot(x+x)=x$ and hence $x\leqslant x$.\\
$\bullet$ Antisymmetry:  for all $x,y\in B$, if $x\leqslant y$ and $y\leqslant x$ then $x+y=y$ and $x+y=x$; thus, $x=y$.\\
$\bullet$ Transitivity:  for all $x,y,z\in B$, if $x\leqslant y$ and $y\leqslant z$, then $x+y=y$ and $y+z=z$. Thus, $x+z=x+(y+z)=(x+y)+z=y+z=z$ which implies that $x\leqslant z$.
\end{proof}
\begin{Proposition}\label{PartialOrderProperties}
Let $(B,+,\cdot,\neg)$ be a Boolean algebra and $x,y,x',y'\in B$. Then:\\
$(a)$ $0\leqslant x\leqslant 1$.\\
$(b)$ $x\cdot y\leqslant x\leqslant x+y$ and $x\cdot y\leqslant y\leqslant x+y$.\\
$(c)$ If $x\leqslant y$ and $x'\leqslant y'$, then $x\cdot x'\leqslant y\cdot y'$.\\
$(d)$ If $x\leqslant y$ and $x'\leqslant y'$, then $x+x'\leqslant y+y'$.\\
$(e)$ $\sup\{x,y\}=x+y$.\\
$(f)$ $\inf\{x,y\}=x\cdot y$.
\end{Proposition}
\begin{proof}
$(a)$ Let $x\in B$, since $x+0=x$ and $x+1=1$, so $0\leqslant x\leqslant 1$. \\
$(b)$ Let $x,y\in B$, since $x+(x\cdot y)=x$ and $x+(x+y)=2x+y=x+y$, hence $x\cdot y\leqslant x\leqslant x+y$. Similarly,  $x\cdot y\leqslant y\leqslant x+y$.\\
$(c)$ Let $x,y,x',y'\in B$ and assume that $x\leqslant y$ and $x'\leqslant y'$. Now, $y\cdot y'\leqslant x\cdot x'+ y\cdot y'\leqslant x\cdot x'+x\cdot y'+ y\cdot x'+ y\cdot y'=x\cdot(x'+y')+y\cdot(x'+y')=(x+y)\cdot (x'+y')=y\cdot y'$ and hence $x\cdot x'+ y\cdot y'=y\cdot y'$ as desired.\\
($d$) Let $x,y,x',y'\in B$ and assume that $x\leqslant y$ and $x'\leqslant y'$, then $(x+x')+(y+y')=(x+y)+(x'+y')=y+y'$.\\
($e$) Let $x,y,z\in B$ such that $x\leqslant z$ and $y\leqslant z$; then by part $(b)$, $x+y\leqslant 2z=z$.\\
($f$) Let $x,y,z\in B$ such that $z\leqslant x$ and $z\leqslant y$; then by part $(c)$, $z=z^2\leqslant x\cdot y$.\\
\end{proof}
\ref{PartialOrderProperties} tells us that every Boolean algebra is a distributive lattice.
\begin{Definition}
Let $(B,+,\cdot,\neg)$ and $(B',+',\cdot',\neg')$ be Boolean algebras. A map\\ $\varphi:B\to B'$ is called a \textbf{homomorphism} if for all $x,y\in B$:\\
$\bullet$ $\varphi(x+y)=\varphi(x)+'\varphi(y)$,\\
$\bullet$ $\varphi(x\cdot y)=\varphi(x)\cdot'\varphi(y)$,\\
$\bullet$ $\varphi(0_B)=0_{B'}$ and $\varphi(1_B)=1_{B'}$.
\end{Definition}
Notice that for $x\in B$, since $0_{B'}=\varphi(0_B)=\varphi(x\cdot (\neg x))=\varphi(x)\cdot'\varphi(\neg x)$ and $1_{B'}=\varphi(1_B)=\varphi(x+(\neg x))=\varphi(x)+'\varphi(\neg x)$, hence by \ref{ComplementProperties}, $\varphi(\neg x)=\neg'\varphi(x)$. 
\begin{Example}
The field $\mathbb{Z}_2$ is a Boolean algebra where the complement $\neg$ of its elements is given by: $\neg 0:=1$ and $\neg 1:=0$. Moreover, $\mathbb{Z}_2$ is unique up to isomorphism; that is, any Boolean algebra with cardinality $2$ is isomorphic to $\mathbb{Z}_2$. $\mathbb{Z}_2$ is called the \textbf{two-element Boolean algebra}.
\end{Example}
Now we associate a Boolean algebra to any given topological space with the property that any continuous map between two topological spaces induces a homomorphism between their associated Boolean algebras; and if the two topological spaces are homeomorphic, their associated Boolean algebras are isomorphic.
\begin{Proposition}\label{ClopBoolean}
Let $X$ be a topological space. Then:\\
($a$) If $f:X\to Y$ is a continuous map between topological spaces, it induces a homomorphism $\widetilde{f}:Clop(Y)\to Clop(X)$.\\
($b$) If $f:X\to Y$ is a homeomorphism between topological spaces, it induces an isomorphism $\widetilde{f}:Clop(Y)\to Clop(X)$.\\
\end{Proposition}
\begin{proof}
($a$) Let $f:X\to Y$ be a continuous map between topological spaces, and define $\widetilde{f}:Clop(Y)\to Clop(X)$ by $\widetilde{f}(U)=f^{-1}(U)$ for all $U\in Clop(Y)$. Clearly, the map $\widetilde{f}:Clop(Y)\to Clop(X)$ is a homomorphism.\\
($b$) Let $g:Y\to X$ be a continuous map such that $g\circ f=id_X$ and $f\circ g=id_Y$. By part ($b$),  $\widetilde{g}:Clop(X)\to Clop(Y)$ is a homomorphism. Moreover, $\widetilde{f}\circ \widetilde{g}=id_{Clop(X)}$ and $\widetilde{g}\circ \widetilde{f}=id_{Clop(Y)}$.
\end{proof}

The clopen Boolean algebra associated with a topological space $X$ is the two-elements Boolean algebra if and only if $X$ is connected. The more components $X$ has, the more interesting the Boolean algebra $Clop(X)$ is. Hausdorff compact spaces that are totally disconnected are of a particular interests as every Boolean algebra is the clopen Boolean algebra associated with a such topological space. 
\section{Stone's space of a Boolean algebra}
In this section, out of a given Boolean algebra $B$, we construct a compact Hausdorff totally disconnected topological space called the \textbf{Stone's space associated with} the Boolean algebra $B$ such that every homomorphism between two Boolean algebras induces a continuous map between their associated Stone's spaces. One way to construct a such space is by taking the set of all ultrafilters of $B$ as the underlying set of the Stone's space associated with a Boolean algebra $B$. Another way to construct the Stone's space associated with $B$ is by taking the family of all homomorphisms from $B$ to $\mathbb{Z}_2$ as the underlying set of the Stone's space associated with a Boolean algebra $B$. In section 5, it will be shown that the Stone's space associated with $B$ can also be obtained as the spectrum of the Boolean ring induced from $B$. We first give a characterization of a totally disconnected space $X$, when $X$ is locally compact.

\subsection{Characterization of locally compact totally disconnected spaces}
\begin{Definition}
Let $X$ a Hausdorff topological space:\\
$\bullet$ $X$ is called \textbf{totally disconnected} if the connected components of $X$ are exactly the singletons. Obviously, every discrete topological space is totally disconnected; however, the converse does not hold. $\mathbb{Q}$ as a subspace of $\mathbb{R}$ serves as an example of a totally disconnected space that is not discrete. \\
$\bullet$ $X$ is called a \textbf{Stone's space} if it is compact and totally disconnected.\\
$\bullet$ $X$ is called \textbf{zero-dimensional} if $X$ admits a basis consists of clopen subsets.
\end{Definition}
\begin{Definition}
Let $X$ be a non-empty set and $\mathcal{A}\subset \mathcal{P}(X)$. $\mathcal{A}$ is said to  \textbf{satisfy the finite intersection property} if the intersection of any finitely many elements of $\mathcal{A}$ is non-empty. It is well-known that a topological space is compact if and only if $\bigcap\limits_{A\in \mathcal{A}}A\neq \emptyset$ for every family of closed subsets $\mathcal{A}$ that satisfies the finite intersection property.
\end{Definition}
\begin{Proposition}\label{Zero-dimensionalIsAlwysTotallyDisconnected}
Let $X$ be a topological space. If $X$ is zero-dimensional then it is totally disconnected.
\end{Proposition}
\begin{proof}
Let $A\subset X$ and $x,y\in A$ be distinct elements. Since $X$ is a $T_1$ space, we can find a clopen subset $U$ containing $x$ but not $y$. Let $V:=X\setminus U$, then $U$ and $V$ are disjoint open subsets containing $x$ and $y$, respectively, and $X=U\cup V$. 
\end{proof}

When $X$ is locally compact, the concept of being totally disconnected coincide with the concept of being zero-dimensional. 

\begin{Lemma}\label{ComponentsOfCompactSpace}
Let $X$ be a Hausdorff compact space. For $x\in X$, let $C_x$ be the component of $X$ containing $x$, and $\mathcal{A}_x$ be the family of all clopen subsets of $X$ containing $x$. Then, $C_x=\bigcap\limits_{A\in \mathcal{A}_x}A$.
\end{Lemma}
\begin{proof}
Notice that every component of $X$ that intersects a clopen subset of $X$ is contained in that clopen subset, so $C_x\subset\bigcap\limits_{A\in \mathcal{A}_x}A$. To prove that $\bigcap\limits_{A\in \mathcal{A}_x}A\subset C_x$, it suffices to show that $\bigcap\limits_{A\in \mathcal{A}_x}A$ is connected. Put $O_x:=\bigcap\limits_{A\in \mathcal{A}_x}A$ and assume that $U_1$ and $U_2$ are disjoint closed subsets of $O_x$ such that $O_x=U_1\cup U_2$. We aim to show that either $U_1=\emptyset$ or $U_2=\emptyset$. Since $X$ is normal, there exist disjoint open subsets $V_1$ and $V_2$ of $X$ such that $U_1\subset V_1$ and $U_2\subset V_2$. Put $V:=V_1\cup V_2$; we claim that there exist finitely many elements $A_1,...,A_n\in \mathcal{A}_x$ such that $\bigcap\limits_{i=1}^{n}A_i\subset V$. Seeking contradiction, assume that the intersection of any finitely many elements of $\mathcal{A}_x$ non-trivially intersects the closed subset $X\setminus V$. The family $\mathcal{A}_x\cup \{X\setminus V\}$ consists of closed subsets of the compact space $X$ and it satisfies the Finite Intersection Property, and hence $O_x\cap (X\setminus V)\neq\emptyset$ which contradicts the fact that $O_x\subset V$. Let, $A_1,...,A_n\in \mathcal{A}_x$ such that $\bigcap\limits_{i=1}^{n}A_i\subset V=V_1\cup V_2$. Without loss of generality, assume that $x\in V_1$, and notice that $A:=V_1\cap \bigcap\limits_{i=1}^{n}A_i$ is an open subset of $X$. Furthermore, as $A^c=(\bigcap\limits_{i=1}^{n}A_i)^c\cup (V_1^c\cap V_2)=(\bigcap\limits_{i=1}^{n}A_i)^c\cup V_2$ is open, $A=V_1\cap \bigcap\limits_{i=1}^{n}A_i$ is a clopen subset of $X$ containing $x$, and hence $A\in \mathcal{A}_x$. This shows that $O_x\subset A\subset V_1$ which implies that $V_2=\emptyset$ and, a fortiori, $U_2=\emptyset$. Thus, $O_x$ is connected as desired. 
\end{proof}
\begin{Proposition}\label{CompactTotallyDisconnectedIsZero-Dimensional}
Let $X$ be a compact totally disconnected space. Then $X$ is zero-dimensional.
\end{Proposition}

\begin{proof}
Let $C_x$ and $\mathcal{A}_x$ be as in \ref{ComponentsOfCompactSpace}. Notice that $C_x=\bigcap\limits_{A\in \mathcal{A}_x}A=\{x\}$. Let $U$ be an open neighborhood of $x\in X$. We can find finitely many elements $A_1,...,A_n\in \mathcal{A}_x$ such that $\bigcap\limits_{i=1}^{n}A_i\subset U$. Since $\bigcap\limits_{i=1}^{n}A_i$ is a clopen subset containing $x$, $X$ admits a basis consists of clopen subsets.
\end{proof}
\begin{Theorem}\label{CharOfTotallyDisconnectedSpaces}
Let $X$ be a Hausdorff locally compact space. $X$ is totally disconnected if and only if $X$ is zero-dimensional.
\end{Theorem}
\begin{proof}
Assume that $X$ is totally disconnected. For $x\in X$, let $U$ be an open neighborhood of $x\in X$. Since $X$ is a Hausdorff locally compact space, we can find a relatively compact open neighborhood $V$ of $x$ such that $V\subset cl(V)\subset U$. Now, $cl(V)$ itself is compact totally disconnected as a subspace of $X$, and by \ref{CompactTotallyDisconnectedIsZero-Dimensional}, $cl(V)$ is zero-dimensional. Let $A$ be a clopen subset of $cl(V)$ containing $x$, then as $A\subset V$, $A$ is clopen subset of $X$ containing $x$ and contained in $U$. Hence, $X$ is zero-dimensional. The other direction follows from \ref{Zero-dimensionalIsAlwysTotallyDisconnected}.
\end{proof}
\begin{Corollary}
Let $X$ be a Hausdorff compact space. $X$ is a Stone's space if and only if $X$ is zero-dimensional.
\end{Corollary}
\begin{proof}
Apply \ref{CharOfTotallyDisconnectedSpaces}.
\end{proof}

\subsection{Stone's space of a Boolean algebra as the family of its ultrafilters} Now we construct the Stone's space that is associated with a Boolean algebra using its ultrafilters as the building blocks. 
\begin{Definition}
Let $(B,+,\cdot,\neg)$ be a Boolean algebra, and $F$ be a nonempty subset of $B$; $F$ is called a \textbf{filter} of $B$ if:\\
$\bullet$ For all $x,y\in F$, $x\cdot y\in F$.\\
$\bullet$ For all $x\in F$ and $y\in B$, if $x\leqslant y$, then $y\in F$.\\
$\bullet$ $F$ is called a \textbf{proper filter} if it is a filter contained properly in $B$. Clearly, a filter $F$ is proper if and only if $0\notin F$.
\end{Definition}
\begin{Definition}
Let $(B,+,\cdot,\neg)$ be a Boolean algebra, and $A\subset B$, $A$ is said to \textbf{satisfy the finite product property} if for any finitely many elements $x_1,...,x_n\in A$, then $x_1\cdot ...\cdot x_n\neq 0$.
\end{Definition}
In the Boolean algebra $(\mathcal{P}(X),\cup,\cap,\neg)$, a subset of $\mathcal{P}(X)$ satisfies the finite product property if and only if it satisfies the finite intersection property. Furthermore, $F\in \mathcal{P}(X)$ is a filter of the Boolean algebra $\mathcal{P}(X)$ if and only if it is a filter in the usual sense.

It can be evidently seen that every subset of a proper filter must satisfy the finite product property. Moreover, every subset of of a Boolean algebra that satisfies the finite product property is contained in a proper filter. 
\begin{Proposition}\label{SetsSFPPContainedInProperFilter}
Let $(B,+,\cdot,\neg)$ be a Boolean algebra and $A$ be a nonempty subset of $B$ that satisfies the finite product property. Then, $A$ is contained in a proper filter $F_A$. Furthermore, if $F$ is a filter containing $A$, then $F_A\subset F$.
\end{Proposition}
\begin{proof}
Set $F_A:=\{x\in B:x_1\cdot ...\cdot x_n\leqslant x$ for some $n\in\mathbb{N}$ and $x_1,...,x_n\in A\}$. Then:\\
$\bullet$ $F_A\neq\emptyset$ as $1\in F_A$.\\
$\bullet$ If $x,y\in F_A$, then there exist $n,m\in\mathbb{N}$ and $x_1,...,x_n,y_1,...,y_n\in A$ such that $x_1\cdot ...\cdot x_n\leqslant x$ and $y_1\cdot ...\cdot y_m\leqslant y$. By \ref{PartialOrderProperties}, $x_1\cdot ...\cdot x_n\cdot y_1\cdot ...\cdot y_m\leqslant x\cdot y$ and hence $x\cdot y\in F_A$.\\
$\bullet$ By the definition of $F_A$, if $x\in F_A$ and $y\in B$ such that $x\leqslant y$, then $y\in F_A$.\\ Finally, it is clear that if $F_A$ is contained in every filter containing $A$.
\end{proof}
\begin{Definition}
Let $(B,+,\cdot,\neg)$ be a Boolean algebra and $A$ be a nonempty subset of $B$ that satisfies the finite product property. The proper filter $F_A$ is called the \textbf{filter generated by} $A$. When $A=\{x\}$ is a singleton, $F_x$ is called the \textbf{principal filter generated by} $\{x\}$. 
\end{Definition}

\begin{Definition}
Let $(B,+,\cdot,\neg)$ be a Boolean algebra, and $F$ be a proper filter of $B$;\\
$\bullet$ $F$ is called a \textbf{prime filter} if for all $x,y\in B$ with $x+y\in F$, then either $x\in F$ or $y\in F$.\\
$\bullet$ $F$ is called a \textbf{maximal filter} if it is maximal with respect to the inclusion.\\
$\bullet$ $F$ is called an \textbf{ultrafilter} if for all $x\in B$, either $x\in F$ or $\neg x\in F$ and not both.\\
$\bullet$ The family of all ultrafilters of $B$ is denoted by $\mathcal{U}(B)$.
\end{Definition}
\begin{Lemma}\label{ProperFilterIsContainedInMaximal}
Let $(B,+,\cdot,\neg)$ be a Boolean algebra, and $F'$ be a proper filter of $B$. For every $x\in B\setminus F$, there exists a maximal filter $M$ of $B$ containing $F'$ such that $x\notin M$. 
\end{Lemma}
\begin{proof}
The family $\mathcal{F}:=\{F\subset B:$ F is a proper filter containing $F'$ and $x\notin F$ $\}$ is partially ordered by the inclusion. Let $\mathcal{C}\subset \mathcal{F}$ be a chain; then $E=\bigcup\limits_{F\in \mathcal{C}}F$ is an upper bound of $\mathcal{C}$ in $\mathcal{F}$. By Zorn's Lemma, $\mathcal{F}$ contains a maximal element $M$. 
\end{proof}
\begin{Lemma}\label{UltrIsPrime}
Let $(B,+,\cdot,\neg)$ be a Boolean algebra, and $F$ be a proper filter of $B$. $F$ is an ultrafilter if and only if it is a prime filter.
\end{Lemma}
\begin{proof}
Assume that $F$ is an ultrafilter of $B$. Let $x,y\in B\setminus F$. Since $F$ is an ultrafilter, one has $\neg x,\neg  y\in F$. Seeking contradiction, assume that $x+y\in F$, then $(\neg x)\cdot(x+y)=(\neg x)\cdot y\in F$ and hence $((\neg x)\cdot y)\cdot(\neg y)=0\in F$, a contradiction.

Conversely, If $F$ is prime and $x\in B\setminus F$; then as $x+(\neg x)=1\in F$, $\neg x$ must be in $F$ and hence $F$ is an ultrafilter.
\end{proof}
\begin{Lemma}\label{UltrIsMax}
Let $(B,+,\cdot,\neg)$ be a Boolean algebra, and $F$ be a proper filter of $B$; $F$ is an ultrafilter if and only if it is a maximal filter.
\end{Lemma}
\begin{proof}
Assume that $F$ is an ultrafilter of $B$. Seeking contradiction, assume that $F$ is not a maximal filter. Let $M$ be a maximal filter containing $F$ and let $x\in M\setminus F$. Since $F$ is an ultrafilter, $\neg x\in F\subset M$. Thus, $x,\neg x\in M$ and in particular $x\cdot \neg x=0\in M$, a contradiction.

Conversely, let $F$ be maximal and $x\in B\setminus F$; notice that if $a\cdot x\neq 0$ for all $a\in F$, then $A:=F\cup \{x\}$ satisfies the finite product property. By \ref{SetsSFPPContainedInProperFilter}, $F$ is contained in a proper filter $F_A$. Since $F$ is maximal, $F=F_A$ which contradicts that fact that $x\in B\setminus F$. Therefore, there exists $a\in F$ such that $a\cdot x= 0$. Now, $a=a\cdot 1=a\cdot(x+(\neg x))=a\cdot (\neg x)\in F$, and since $a\cdot (\neg x)\leqslant \neg x$, $\neg x$ must be in $F$ and hence $F$ is an ultrafilter.
\end{proof}
\begin{Corollary}\label{UltraAreMaxSatTheFPP}
Let $(B,+,\cdot,\neg)$ be a Boolean algebra, and $F$ be an ultrafilter of $B$; if $x\in B$ such that $x\cdot a\neq 0$ for all $a\in F$, then $x\in F$.
\end{Corollary}
\begin{proof}
Notice that $A:=F\cup\{x\}$ satisfies the finite product property. Hence $F_A$ is a proper filer containing $F$. By the maximality of $F$, one has $F=F_A$ as wanted. 
\end{proof}
\begin{Example}\label{PrincipalFilterAreUltra}
Consider the power set Boolean algebra $(\mathcal{P}(X),\cup,\cap,\neg)$ and $x\in X$; the principal filter $F_x$ generated by $x$ is an ultrafilter. Moreover, as ultrafilters are maximal, if $F$ is an ultrafilter containing $\{x\}$, then $F_x=F$.
\end{Example}
\begin{Lemma}\label{HomoAndUlta}
Let $(B,+,\cdot,\neg)$ be a Boolean algebra, and $F$ be a proper filter of $B$; $F$ is an ultrafilter if and only if $F=\varphi_{{\textstyle\mathstrut}F}^{-1}\{1\}$ for some homomorphism $\varphi_{{\textstyle\mathstrut}F}:B\to \mathbb{Z}_2$.
\end{Lemma}
\begin{proof}
Assume that $F$ is an ultrafilter, and define $\varphi_{{\textstyle\mathstrut}F}:B\to \mathbb{Z}_2$ by:\\
\begin{equation*}
\varphi_{{\textstyle\mathstrut}F}(x)=\begin{cases}
          0 \quad &\text{if} \, x \notin F \\
          1\quad &\text{if} \, x \in F \\
     \end{cases}
\end{equation*}
We claim that $\varphi_{{\textstyle\mathstrut}F}:B\to \mathbb{Z}_2$ is a homomorphism.\\
$\bullet$ Clearly, $\varphi_{{\textstyle\mathstrut}F}(0_B)=0$ and $\varphi_{{\textstyle\mathstrut}F}(1_B)=1$ as $1_B\in F$ and $0_B\notin F$.\\
$\bullet$ Let $x,y\in B$; on one hand, if $x+y\in F$ then either $x\in F$ or $y\in F$. Without loss of generality assume that $x\in F$, then $\varphi_{{\textstyle\mathstrut}F}(x)=1$ and $\varphi_{{\textstyle\mathstrut}F}(x)+\varphi_{{\textstyle\mathstrut}F}(y)=1+\varphi_{{\textstyle\mathstrut}F}(y)=1$. Hence, $1=\varphi_{{\textstyle\mathstrut}F}(x+y)=\varphi_{{\textstyle\mathstrut}F}(x)+\varphi_{{\textstyle\mathstrut}F}(y)$. On the other hand, if $x+y\notin F$ then $x\notin F$ and $y\notin F$. Hence, $0=\varphi_{{\textstyle\mathstrut}F}(x+y)=\varphi_{{\textstyle\mathstrut}F}(x)+\varphi_{{\textstyle\mathstrut}F}(y)$.\\
$\bullet$ Let $x,y\in B$; on one hand, if $x\cdot y\in F$, then $x\in F$ and $y\in F$. Hence, $1=\varphi_{{\textstyle\mathstrut}F}(x\cdot y)=\varphi_{{\textstyle\mathstrut}F}(x)\cdot \varphi_{{\textstyle\mathstrut}F}(y)$. On the other hand, if $x\cdot y\notin F$, then either $x\notin F$ or $y\notin F$. Without loss of generality assume that $x\notin F$, then $0=\varphi_{{\textstyle\mathstrut}F}(x\cdot y)=0\cdot \varphi_{{\textstyle\mathstrut}F}(y)=\varphi_{{\textstyle\mathstrut}F}(x)\cdot \varphi_{{\textstyle\mathstrut}F}(y)$. Therefore, $\varphi_{{\textstyle\mathstrut}F}:B\to \mathbb{Z}_2$ is a homomorphism and $F=\varphi_{{\textstyle\mathstrut}F}^{-1}\{1\}$.

Conversely, let  $\varphi_{{\textstyle\mathstrut}F}:B\to \mathbb{Z}_2$ is a homomorphism, and let $F:=\varphi_{{\textstyle\mathstrut}F}^{-1}\{1\}$. Clearly, $F$ is a proper filter. If $x\in B\setminus F$, then $\varphi_{{\textstyle\mathstrut}F}(\neg x)=\neg \varphi_{{\textstyle\mathstrut}F}(x)=\neg 0=1$ and hence $\neg x\in F$. 
\end{proof}
\begin{Definition}
Let $(B,+,\cdot,\neg)$ be a Boolean algebra; and $I\subset B$:\\
$\bullet$  $I$ is called an \textbf{ideal} if for all $a,b\in I$ and $x\in R$, then $x\cdot a, a+b\in I$.\\
$\bullet$  A proper ideal $I$ is called a \textbf{prime ideal} if for all $x,y \in B$ with $x\cdot y\in I$, then either $x\in I$ or $x\in I$.\\
$\bullet$ A proper ideal $I$ is called a \textbf{maximal ideal} if it is maximal with respect the inclusion.\\
\end{Definition}
\begin{Lemma}\label{HomoAndPrime}
Let $(B,+,\cdot,\neg)$ be a Boolean algebra, and $I$ be a proper ideal of $B$; $I$ is a prime ideal if and only if $I=\varphi_{{\textstyle\mathstrut}I}^{-1}\{0\}$ for some homomorphism $\varphi_{{\textstyle\mathstrut}I}:B\to \mathbb{Z}_2$.
\end{Lemma}
\begin{proof}
Assume that $I$ is a prime ideal, and define $\varphi_{{\textstyle\mathstrut}I}:B\to \mathbb{Z}_2$ by:\\
\begin{equation*}
\varphi_{{\textstyle\mathstrut}I}(x)=\begin{cases}
          0 \quad &\text{if} \, x \in I \\
          1\quad &\text{if} \, x \notin I \\
     \end{cases}
\end{equation*}
We claim that $\varphi_{{\textstyle\mathstrut}I}:B\to \mathbb{Z}_2$ is a homomorphism.\\
$\bullet$ Clearly, $\varphi_{{\textstyle\mathstrut}I}(0_B)=0$ and $\varphi_{{\textstyle\mathstrut}I}(1_B)=1$ as $1_B\notin I$ and $0_B\in I$.\\
$\bullet$ Let $x,y\in B$; notice that $x+y\in I$ if and only if $x,y\in I$. Indeed, if $x\notin I$ and $x+y\in I$, then by the absorption law one has $x\cdot(x+y)=x+x\cdot y=x\in I$, a contradiction. Now, if $x+y\in I$ then $x,y\in I$ and hence $0=\varphi_{{\textstyle\mathstrut}I}(x+y)=\varphi_{{\textstyle\mathstrut}I}(x)+\varphi_{{\textstyle\mathstrut}I}(y)$. If $x+y\notin I$ then $x,y\notin I$ and hence  $1=\varphi_{{\textstyle\mathstrut}I}(x+y)=\varphi_{{\textstyle\mathstrut}I}(x)+\varphi_{{\textstyle\mathstrut}I}(y)$ as $1+1=1$.\\
$\bullet$ Let $x,y\in B$; if either $x\in I$ or $y\in I$, then $x\cdot y\in I$ and hence $0=\varphi_{{\textstyle\mathstrut}I}(x\cdot y)=\varphi_{{\textstyle\mathstrut}I}(x)\cdot\varphi_{{\textstyle\mathstrut}I}(y)$. If $x,y\notin I$, then as $I$ is prime $x\cdot y\notin I$ and hence $1=\varphi_{{\textstyle\mathstrut}I}(x\cdot y)=\varphi_{{\textstyle\mathstrut}I}(x)\cdot\varphi_{{\textstyle\mathstrut}I}(y)$. 

Conversely, let  $\varphi_{{\textstyle\mathstrut}I}:B\to \mathbb{Z}_2$ is a homomorphism such that $I:=\varphi_{{\textstyle\mathstrut}I}^{-1}\{0\}$. Clearly, $I$ is a proper ideal. Seeking contradiction, assume that $x,y\in B\setminus I$ such that $x\cdot y\in I$, then  $0=\varphi_{{\textstyle\mathstrut}I}(x\cdot y)=\varphi_{{\textstyle\mathstrut}I}(x)\cdot \varphi_{{\textstyle\mathstrut}I}(y)=1$, a contradiction. 
\end{proof}
\begin{Corollary}
Let $(B,+,\cdot,\neg)$ be a Boolean algebra, and $F\subset B$. $F$ is an ultrafilter of $B$ if and only if $B\setminus F$ is a prime ideal.
\end{Corollary}
\begin{proof}
Apply \ref{HomoAndUlta} and \ref{HomoAndPrime}.
\end{proof}
From \ref{UltrIsPrime}, \ref{UltrIsMax}, \ref{HomoAndUlta} and \ref{HomoAndPrime}, we have:
\begin{Corollary}
Let $(B,+,\cdot,\neg)$ be a Boolean algebra, and $F\subset B$ be a proper filter. The following are equivalent:\\
($a$) $F$ is an ultrafilter.\\
($b$) $F$ is a prime filter.\\
($c$) $F$ is a maximal filter.\\
($d$) $B\setminus F$ is a prime ideal.
\end{Corollary}

\begin{Notation}
Let $B$ be a Boolean algebra and $x\in B$, we put: $$U_x=\{F\in\mathcal{U}(B):x\in F\}.$$ 
\end{Notation}
\begin{Lemma}\label{PropertiesOfBasisElements}
Let $(B,+,\cdot,\neg)$ be a Boolean algebra. Then:\\
($a$) For all $x,y\in B$, $U_x\cap U_y=U_{x\cdot y}$.\\
($b$) If $x_1,...,x_n\in B$, then $\bigcap\limits_{i=1}^n U_{x_i}=\emptyset$ if and only if $x_1\cdot...\cdot x_n=0$.\\
($c$) For all $x,y\in B$, $U_x\cup U_y=U_{x+y}$.\\
($d$) For all $x\in B$, $\mathcal{U}(B)\setminus U_x=U_{\neg x}$.\\
($e$) The family $\{U_x:x\in B\}$ is a basis for a topology $\mathcal{T}_{\mathcal{U}(B)}$ on $\mathcal{U}(B)$ in which $U_x$ is a clopen subset for all $x\in B$. 
\end{Lemma}
\begin{proof}
($a$) Let $x,y\in B$, then $F\in U_x\cap U_y$ if and only if $x\cdot y\in F$.\\
($b$) By \ref{SetsSFPPContainedInProperFilter}, \ref{ProperFilterIsContainedInMaximal} and \ref{UltrIsMax}, $U_{x}=\emptyset$ if and only if $x=0$. Applying part ($a$), $\bigcap\limits_{i=1}^n U_{x_i}=\emptyset$ if and only if $x_1\cdot...\cdot x_n=0$.\\
($c$) Let $x,y\in B$. If $F\in U_x\cup U_y$, then either $x\in F$ or $y\in F$; in either case, $x+y$ must be in $F$. Conversely,  if $F\in U_{x+y}$, then $x+y\in F$ and hence either $x\in F$ or $y\in F$ as $F$ is prime.\\
($d$) Let $x\in B$, as $F$ is an ultrafilter then $x\notin F$ if and only if  $\neg x\in F$.\\
($e$) The family $\{U_x:x\in B\}$ being basis follows from part ($a$) and the fact that $U_1=\mathcal{U}(B)$. Part ($c$) shows that $U_x$ is a clopen subset for all $x\in B$.
\end{proof}
\begin{Theorem}\label{SpaceOfUltraOfBoolean}
Let $B$ be a Boolean algebra, the topological space $(\mathcal{U}(B),\mathcal{T}_{\mathcal{U}(B)})$ is a Stone's space.
\end{Theorem}
\begin{proof}
$\bullet$ $\mathcal{U}(B)$ is Hausdorff. Indeed, let $F_1,F_2\in\mathcal{U}(B)$ be distinct ultrafilters and let $x\in F_1\setminus F_2$, then $\neg x\in F_2$ and hence $F_1$ and $F_2$ are contained in the disjoint open sets $U_x$ and $U_{\neg x}$, respectively.\\
$\bullet$ $\mathcal{U}(B)$ is compact. To this end, it suffices to show that for arbitrary collection $\mathcal{C}$ of closed sets of $\mathcal{U}(B)$ that satisfies the finite intersection property, then $\bigcap\limits_{C\in\mathcal{C}} C\neq\emptyset$. Notice that for every $C\in\mathcal{C}$, there exists $A_{C}\subset B$ such that $C=\bigcap\limits_{x\in A_{C}} U_{x}$. Put $A:=\bigcup\limits_{C\in\mathcal{C}}A_C$, and notice that as $\mathcal{C}$ satisfies the finite intersection property, the family $\{U_x:x\in A\}$ satisfies the finite intersection property and hence, by \ref{PropertiesOfBasisElements} part ($b$), $A$ satisfies the finite product property. Let $F_A$ be an ultrafilter of $B$ containing $A$. Therefore, $F_A\in \bigcap\limits_{x\in A} U_{x}=\bigcap\limits_{C\in\mathcal{C}} C$.\\
$\bullet$ Since $(\mathcal{U}(B),\mathcal{T}_{\mathcal{U}(B)})$ is locally compact and it admits a basis consists of clopen subsets, by \ref{CharOfTotallyDisconnectedSpaces} it is totally disconnected. 
\end{proof}

It is worth mentioning that if $G$ is a group acts on $X$; the Stone's space of the Boolean algebra $\faktor{Comm_G(X)}{\sim}$, given in \ref{SpaceOEnds}, is called \textbf{the space of ends} of $X$ \cite{Corn}. 

\begin{Notation}
For $A\subset X$, the family of all principal filters $F_a$ such that $a\in A$ is denoted by $Pr(A)$; that is $Pr(A):=\{F_a\in \mathcal{U}(\mathcal{P}(X)):a\in A\}$. 
\end{Notation}
\begin{Proposition}
Let $F\in \mathcal{U}(\mathcal{P}(X))$ and $A\in \mathcal{P}(X)$. $A\in F$ if and only $F\in cl(Pr(A))$ where the closure is taking in the topological space $\mathcal{U}(\mathcal{P}(X))$.
\end{Proposition}
\begin{proof}
Assume that $A\in F$, and let $U_B$ be a neighborhood of $F$. Since $A,B\in F$, $A\cap B\neq\emptyset$. Pick $a\in A\cap B$, then $F_a\in U_B\cap Pr(A)$.

Conversely, assume that $F\in cl(Pr(A))$. By \ref{UltraAreMaxSatTheFPP}, to show that $A\in F$, it suffices to show that for every $B\in F$, $A\cap B\neq\emptyset$. Indeed, if $B\in F$, then $F\in U_B$ and hence $U_B\cap Pr(A)\neq\emptyset$. Therefore, there exists $a\in A$ such that $F_a\in U_B$, and hence $a\in A\cap B$ as desired. 
\end{proof}

\begin{Corollary}\label{ClopenSetsAreBasisElements}
Let $B$ be a Boolean algebra and $C$ be a clopen subset of $\mathcal{U}(B)$. There exists $x\in B$ such that $C=U_x$, where $U_x=\{F\in\mathcal{U}(B):x\in F\}$.
\end{Corollary}
\begin{proof}
On one hand, $C$ is an open subset of $\mathcal{U}(B)$, so $C$ there exists a subset $A\subset B$ such that $C=\bigcup\limits_{x\in A} U_{x}$. On the other hand, $C$ is a closed subset of the Hausdorff compact space $\mathcal{U}(B)$ and hence it is compact. In particular, there exist $x_1,...,x_n\in A$ such that $C=\bigcup\limits_{i=1}^n U_{x_i}$. Put $x=x_1+...+x_n$, then $C=U_x$ as desired. 
\end{proof}

\begin{Theorem}\label{IsoBooleansHaveHomeoStonesSpaces}
Every homomorphism $\varphi:B_1\to B_2$ between Boolean algebras induces a continuous map $\widetilde{\varphi}:\mathcal{U}(B_2)\to \mathcal{U}(B_1)$ between their associated Stone's spaces. If, moreover, $\varphi:B_1\to B_2$ is an isomorphism, then its induced map $\widetilde{\varphi}:\mathcal{U}(B_2)\to \mathcal{U}(B_1)$ is a homeomorphism.
\end{Theorem}
\begin{proof}
First notice that if $F$ is an ultrafilter of $B_2$, then $\varphi^{-1}(F)$ is an ultrafilter of $B_1$. Consider the map $\widetilde{\varphi}:\mathcal{U}(B_2)\to \mathcal{U}(B_1)$ that sends $F\in \mathcal{U}(B_2)$ to $\varphi^{-1}(F)$. Now, if $U_x$ is a clopen subset of $\mathcal{U}(B_1)$, then $\widetilde{\varphi}^{-1}(U_x)=U_{\varphi(x)}$ as $F\in \widetilde{\varphi}^{-1}(U_x)$ if and only if $\widetilde{\varphi}(F)\in (U_x)$ if and only if $\varphi^{-1}(F)\in (U_x)$ if and only if $\varphi(x)\in F$, an hence $\widetilde{\varphi}:\mathcal{U}(B_2)\to \mathcal{U}(B_1)$ is continuous. It is straightforward to see that $\widetilde{\varphi}:\mathcal{U}(B_2)\to \mathcal{U}(B_2)$  is a homeomorphism when $\varphi:B_1\to B_2$ is an isomorphism.
\end{proof}
\subsection{Stone's space of a Boolean algebra $B$ as the family of all homomorphisms from $B$ to $\mathbb{Z}_2$} The aim of this section is to equip $Hom(B,\mathbb{Z}_2)$ with a topology such that $Hom(B,\mathbb{Z}_2)$ is a Stone's space homeomorphic to $(\mathcal{U}(B),\mathcal{T}_{\mathcal{U}(B)})$. That goal is achieved by embedding $Hom(B,\mathbb{Z}_2)$ in a Stone's space where the image of $Hom(B,\mathbb{Z}_2)$ under that embedding is a closed subset. 
\begin{Lemma}\label{ProductOTheTwoElementIsStone}
Let $X$ be a non-empty set.\\
($a$) The family $2^X:=\{f:X\to \mathbb{Z}_2:$ $f$ is a function $\}$ admits a topology $\mathcal{T}_{2^X}$ such that $(2^X,\mathcal{T}_{2^X})$ is homeomorphic to $\prod\limits _{x \in X}\mathbb{Z}_2$ equipped with the product topology considering $\mathbb{Z}_2$ as a discrete space.\\ 
($b$) The subsets $0_a:=\{f\in 2^X: f(a)=0\}$ and $1_a:=\{f\in 2^X: f(a)=1\}$ are clopen subsets of $(2^X,\mathcal{T}_{2^X})$ for all $a\in X$ and that the family $\mathcal{S}:=\{0_a,1_a:a\in X\}$ is a subbasis for the topology $\mathcal{T}_{2^X}$ on $2^X$.\\
($c$) $2^X$ is a Stone's space.
\end{Lemma}
\begin{proof}
($a$) Define the map $\varphi:2^X\to  \mathbb{Z}_2$ that sends each $f\in 2^X$ to $(f(x))_{x\in X}\in \prod\limits _{x \in X}\mathbb{Z}_2$. Since $\varphi:2^X\to  \mathbb{Z}_2$ is a bijection, the family $\mathcal{T}_{2^X}:=\{\varphi^{-1}(U):U\subset \prod\limits _{x \in X}\mathbb{Z}_2$ is open$\}$ is a topology on $2^X$ such that $(2^X,\mathcal{T}_{2^X})$ is homeomorphic to $\prod\limits _{x \in X}\mathbb{Z}_2$.\\
($b$) Notice that for $a\in X$, a subset of $\prod\limits _{x \in X}\mathbb{Z}_2$ that has the form $\prod\limits _{x \in X}U_x$, where $U_x=\mathbb{Z}_2$ for all $x\neq a$ and $U_x=\{0\}$ or $U_x=\{1\}$, if $x=a$, is a subbasis element and the family of all such subsets consists a subbasis for $\prod\limits _{x \in X}\mathbb{Z}_2$. If $U_a=\{0\}$, then $\varphi^{-1}(\prod\limits _{x \in X}U_x)=0_a$; and if $U_a=\{1\}$, then $\varphi^{-1}(\prod\limits _{x \in X}U_x)=1_a$. Hence, the family $\mathcal{S}:=\{0_a,1_a:a\in X\}$ is a subbasis. Moreover, as $2^X=0_a\cup 1_a$ and $0_a\cap 1_a=\emptyset$ for all $a\in X$, the subsets $0_a:=\{f\in 2^X: f(a)=0\}$ and $1_a:=\{f\in 2^X: f(a)=1\}$ are clopen.\\
($c$) $2^X$ is Hausdorff as $\mathbb{Z}_2$ is Hausdorff. By Tychonoff's Theorem, $2^X$ is  compact. From part($b$), $(2^X,\mathcal{T}_{2^X})$ has a basis consists of clopen subsets. By \ref{CharOfTotallyDisconnectedSpaces}, $2^X$ is a Stone's space.
\end{proof}
\begin{Theorem}\label{HomIsAstone}
Let $Hom(B,\mathbb{Z}_2):=\{f:B\to \mathbb{Z}_2:$ $f$ is a homomorphism$\}$ where $(B,+_{{\textstyle\mathstrut}B},\cdot,\neg)$ is a Boolean algebra. Then  $Hom(B,\mathbb{Z}_2)$ is a closed subset of the Stone's space $2^B$. In particular, $Hom(B,\mathbb{Z}_2)$ is a Stone's space.
\end{Theorem}
\begin{proof}
Let $a\in B$, and consider the projection map $\pi_a:\prod\limits _{x \in X}\mathbb{Z}_2\to \mathbb{Z}_2$ and the map $\varphi:2^X\to  \mathbb{Z}_2$ defined in the proof of \ref{ProductOTheTwoElementIsStone}. Since $\varphi\circ \pi_a:2^B\to \mathbb{Z}_2$ is continuous, then $(\varphi\circ \pi_a)^{-1}\{0\}=\{f\in 2^B:f(a)=0\}$ and  $(\varphi\circ\pi_a)^{-1}\{1\}=\{f\in 2^B:f(a)=1\}$ are clopen for all $a\in B$. For all $a,b\in B$, put, $A_{0}:=\{f\in 2^B:f(0_B)=0\}$, $A_{1}:=\{f\in 2^B:f(1_B)=1\}$, $A_{a+b}:=\{f\in 2^B:f(a+_{{\textstyle\mathstrut}B}b)=f(a)+f(b)\}$, and $A_{a\cdot b}:=\{f\in 2^B:f(a\cdot_{{\textstyle\mathstrut}B} b)=f(a)\cdot f(b)\}$. Notice that $A_{0}$ and $A_{1}$ are closed subsets by \ref{ProductOTheTwoElementIsStone}($b$). Moreover:\\
$\bullet$ $A_{a+b}$ is the union of the closed sets: $$\{f\in 2^B:f(a+_{{\textstyle\mathstrut}B}b)=0\}\cap \{f\in 2^B:f(a)=0\}\cap \{f\in 2^B:f(b)=0\},$$ $$\{f\in 2^B:f(a+_{{\textstyle\mathstrut}B}b)=1\}\cap \{f\in 2^B:f(a)=1\}\cap \{f\in 2^B:f(b)=0\},$$ $$\{f\in 2^B:f(a+_{{\textstyle\mathstrut}B}b)=1\}\cap \{f\in 2^B:f(a)=0\}\cap \{f\in 2^B:f(b)=1\},$$  $$\{f\in 2^B:f(a+_{{\textstyle\mathstrut}B}b)=1\}\cap \{f\in 2^B:f(a)=1\}\cap \{f\in 2^B:f(b)=1\}.$$\\
$\bullet$ $A_{a\cdot b}$ is the union of the closed sets:
$$\{f\in 2^B:f(a\cdot_{{\textstyle\mathstrut}B} b)=0\}\cap \{f\in 2^B:f(a)=0\}\cap \{f\in 2^B:f(b)=0\},$$ $$\{f\in 2^B:f(a\cdot_{{\textstyle\mathstrut}B} b)=0\}\cap \{f\in 2^B:f(a)=1\}\cap \{f\in 2^B:f(b)=0\},$$ $$\{f\in 2^B:f(a\cdot_{{\textstyle\mathstrut}B} b)=0\}\cap \{f\in 2^B:f(a)=0\}\cap \{f\in 2^B:f(b)=1\},$$ $$\{f\in 2^B:f(a\cdot_{{\textstyle\mathstrut}B} b)=1\}\cap \{f\in 2^B:f(a)=1\}\cap \{f\in 2^B:f(b)=1\}.$$\\
Therefore, $Hom(B,\mathbb{Z}_2)=A_0\cap A_1\cap(\bigcap\limits_{a,b\in B}A_{a+b})\cap(\bigcap\limits_{a,b\in B}A_{a\cdot b})$ is a closed subset of the Stone's space $2^B$ and hence it is a Stone's space.
\end{proof}
\begin{Theorem}\label{StoneAsSetOfHom}
Let $(B,+_{{\textstyle\mathstrut}B},\cdot,\neg)$ be Boolean algebra. The Stone's space $Hom(B,\mathbb{Z}_2)$ is homeomorphic to $\mathcal{U}(B)$. 
\end{Theorem}
\begin{proof}
By \ref{HomoAndUlta}, every $f\in Hom(B,\mathbb{Z}_2)$ determines an ultrafilter $F_f=f^{-1}(\{1\})$, and for every ultrafilter $F\in \mathcal{U}(B)$ there is a homomorphism $f_F\in Hom(B,\mathbb{Z}_2)$ such that $f^{-1}_{F}(\{1\})=F$. Define $\varphi:Hom(B,\mathbb{Z}_2)\to \mathcal{U}(B)$ by $\varphi(f):=F_f$ for all $f\in Hom(B,\mathbb{Z}_2)$. Given a basis element $U_x=\{F\in \mathcal{U}(B):x\in F\}$, where $x\in B$, then:
\begin{equation*} 
\begin{split}
\varphi^{-1}(U_x) & = \{f\in Hom(B,\mathbb{Z}_2):F_f\in U_x\} \\
 & = \{f\in Hom(B,\mathbb{Z}_2):x\in F_f=f^{-1}(\{1\})\} \\
 & = \{f\in Hom(B,\mathbb{Z}_2):f(x)=1\}
\end{split}
\end{equation*}
which is a clopen subset of $Hom(B,\mathbb{Z}_2)$, because $\{f\in Hom(B,\mathbb{Z}_2):f(x)=1\}\\=Hom(B,\mathbb{Z}_2)\cap\{f\in 2^B:f(x)=1\}$ and $\{f\in 2^B:f(x)=1\}$ is a clopen subset of $2^B$ as seen in \ref{HomIsAstone}. Therefore, $\varphi:Hom(B,\mathbb{Z}_2)\to \mathcal{U}(B)$ is a bijective continuous map between compact Hausdorff spaces and hence it is a homeomorphism as desired.
\end{proof}
\section{The Stone's representation theorems}
Every Boolean algebra (up to isomorphism) is the clopen Boolean algebra associated with some Stone's space; and conversely, every Stone's space (up to homeomorphism) is the space of all ultrafilters of some Boolean algebra. These tow facts are the core of the Stone's representation theorems. 
\begin{Theorem}\textbf{(Stone's representation theorem for Boolean algebras)}\label{StoneRepThOfBooleanAlg}
Up to isomorphism, every Boolean algebra is the clopen Boolean algebra for some Stone's space. More specifically, if $B$ is a Boolean algebra, then it is isomorphic to the Boolean algebra $Clop(\mathcal{U}(B))$.
\end{Theorem}
\begin{proof}
Let $(B,+,\cdot,\neg )$ be a Boolean algebra; consider the map $\varphi:B\to Clop(\mathcal{U}(B))$ that sends each $x\in B$ to $U_x\in Clop(\mathcal{U}(B))$. It is straightforward to see that $\varphi:B\to Clop(\mathcal{U}(B))$ is a homomorphism. It is, moreover, surjective by \ref{ClopenSetsAreBasisElements}. Now, let $x,y\in B$ be distinct elements, then either $x\nleqslant y$ or $y\nleqslant x$. Without loss of generality, assume that $x\nleqslant y$, then $y\notin F_x$, where $F_x$ is the principal filter generated by $x$. Let $M_x$ be a maximal filter containing  $F_x$ such that $y\notin M_x$. Notice that $M_x\in U_x$ and $M_x\notin U_y$; hence $U_x\neq U_y$. Consequently, $\varphi:B\to Clop(\mathcal{U}(B))$ is an isomorphism as desired.
\end{proof}

\begin{Lemma}\label{UltrafiltersOfTheClopenBoolean}
Let $X$ be a Stone's space. Then:\\
($a$) For every $x\in X$, the family $F^x:=\{U\in Clop(X):x\in U\}$ is an ultrafilter of the clopen Boolean algebra $Clop(X)$.\\
($b$) If $F$ is an ultrafilter of $Clop(X)$, then $\bigcap\limits_{U\in F}U=\{x_F\}$ for some $x_F\in X$. 
\end{Lemma}
\begin{proof}
($a$) Since $X$ is totally disconnected, $Clop(X)$ is a basis for $X$. Since $X$ is Hausdorff $F^x$ is a proper subset of $Clop(X)$. Evidently, $F^x$ is a filter; and if $V\in Clop(X)\setminus F^x$, then $U:=X\setminus V$ is a clopen subset containing $x$ and hence it is contained in $F^x$. Thus, $F^x$ is an ultrafilter.\\
($b$) Let $F$ be an ultrafilter of $Clop(X)$. Since $F$ is a family of closed sets of the compact space $X$ and that it satisfies the finite intersection property; hence $\bigcap\limits_{U\in F}U$ is non-empty. Seeking contradiction, let $x,y\in\bigcap\limits_{U\in F}U$ be distinct elements. Let $U_x$ be a clopen set containing $x$ such that $y\notin U_x$. Notice that $U_y:=X\setminus U_x$ is a clopen subset containing $y$. Since $F$ is an ultrafilter, either $U_x\in F$ or $U_y\in F$. Without loss of generality, assume that $U_x\in F$, then as $y\notin U_x$,  $y\notin\bigcap\limits_{U\in F}U$, a contradiction. 
\end{proof}
\begin{Theorem}\textbf{(Stone's representation theorem for Stone's spaces)}\label{StoneRepThOfStoneSpace}
Up to homeomorphism, every Stone's space is the ultrafilter space of some Boolean algebra. More specifically, if $X$ is a Stone's space then it is homeomorphic to $\mathcal{U}(Clop(X))$. 
\end{Theorem}
\begin{proof}
Let $X$ be a Stone's space. Exploiting \ref{UltrafiltersOfTheClopenBoolean}($a$), we define the map $\psi :X\to \mathcal{U}(Clop(X))$ that sends each $x\in X$ to $F^x$. To show that $\psi :X\to \mathcal{U}(Clop(X))$ is homeomorphism, it suffices to show that it is continuous bijective map as $X$ and $\mathcal{U}(Clop(X))$ are compact Hausdorff spaces.\\
$\bullet$ $\psi$ is injective. Indeed, let $x,y\in X$ be distinct elements. Since $X$ is Hausdorff and $Clop(X)$ is a basis for $X$, there exists $U_x\in Clop(X)$ containing $x$ such that $y\notin U_x$. Hence, $U_x\in F^x\setminus F^y$.\\
$\bullet$ $\psi$ is surjective. Indeed, let $F$ be an ultrafilter of $Clop(X)$, then $\bigcap\limits_{U\in F}U=\{x_F\}$, by \ref{UltrafiltersOfTheClopenBoolean}($b$). Now, by \ref{UltrafiltersOfTheClopenBoolean}($a$), $F^{x_F}$ is an ultrafilter and it is obvious that $F\subset F^{x_F}$. Since ultrafilters are maximal, $F= F^{x_F}$.\\
$\bullet$ $\psi$ is continuous. To this end, notice that a basis element of $\mathcal{U}(Clop(X))$ is of the form $U_{C}=\{F\in \mathcal{U}(Clop(X)):C\in F\}$ for some a clopen subset $C$ of $X$. Now, $\psi^{-1}(U_{C})=\{x\in X:F^x\in U_{C}\}=\{x\in X:C\in F^x\}=C$ and thus $\psi$ is continuous.
\end{proof}
\begin{Corollary}\label{IsoAndHomeo}
($a$) Let $B_1$ and $B_1$ be Boolean algebras; then $B_1$ and $B_1$ are isomorphic if and only if their Stone's spaces $\mathcal{U}(B_1)$ and $\mathcal{U}(B_2)$ are homeomorphic.\\
($b$) Let $X$ and $Y$ be Stone's spaces; then $X$ and $Y$ are homeomorphic if and only if $Clop(X)$ and $Clop(Y)$ are isomorphic.
\end{Corollary}
\begin{proof}
($a$) Assume that the Stone's spaces $\mathcal{U}(B_1)$ and $\mathcal{U}(B_2)$ are homeomorphic. By \ref{ClopBoolean}, $Clop(\mathcal{U}(B_1))$ and $Clop(\mathcal{U}(B_2))$ are isomorphic. By \ref{StoneRepThOfBooleanAlg}, $B_1$ and $B_1$ are isomorphic. The other direction follows from \ref{IsoBooleansHaveHomeoStonesSpaces}. \\
($b$) If $X$ and $Y$ are homeomorphic, then by \ref{ClopBoolean}, $Clop(X)$ and $Clop(Y)$ are isomorphic. Conversely, assume that $Clop(X)$ and $Clop(Y)$ are isomorphic, then by \ref{IsoBooleansHaveHomeoStonesSpaces}, $\mathcal{U}(Clop(X))$ and $\mathcal{U}(Clop(Y))$ are homeomorphic. Therefore, by \ref{StoneRepThOfStoneSpace}, $X$ and $Y$ are homeomorphic. 
\end{proof}
\section{Stone's space of a Boolean algebra as the spectrum of its induced Boolean ring}
Alternatively the Stone's space $(\mathcal{U}(B),\mathcal{T}_{\mathcal{U}(B)})$ can be obtained as the spectrum of some Boolean ring induced from the Boolean algebra $B$.
\begin{Definition}
Let $R$ be a commutative ring with identity $1$; and $I\subset R$:\\
$\bullet$  $I$ is called an \textbf{ideal} if it is closed under addition and $R\cdot I=I$; that is, for all $a,b\in I$ and $r\in R$, then $r\cdot a+b\in I$.\\
$\bullet$  A proper ideal $I$ is called a \textbf{prime ideal} if for all $x,y \in R$ with $x\cdot y\in I$, then either $x\in I$ or $x\in I$.\\
$\bullet$ A proper ideal $I$ is called a \textbf{maximal ideal} if it is maximal with respect to the inclusion.\\
$\bullet$ The family of all prime ideals of $R$ is called the \textbf{spectrum} of $R$ and is denoted by $Spec(R)$.
\end{Definition}
\begin{Lemma}\label{IdealsProperties}
Let $R$ be a commutative ring with identity $1$. Then:\\
($a$) Let $A\subset R$, the collection $I_A:=\{r_{1}\cdot a_{1}+...+r_{n}\cdot a_{n}:n\in\mathbb{N},a_i\in A, r_i\in R\}$ of all finite linear combinations of elements of $A$
is an ideal called the  \textbf{ideal generated} by $A$.\\
($b$) Every proper ideal is contained in a maximal ideal.\\
($c$) Every maximal ideal is prime.
\end{Lemma}
\begin{proof}
($a$) Straightforward.\\
($b$) Let $I$ be a proper ideal of $R$. Apply the Zorn's Lemma to the family $\{J\subset R:J$ is a proper ideal containing $I\}$.\\
($c$) Let $M$ be a maximal ideal and $x,y\in R$ such that $x\cdot y\in M$. Assume that $x\notin M$, then $M$ is contained properly in the ideal $I_A$ generated by $A:=M\cup\{x\}$. By the maximality of $M$, $I_A=R$. Since $1\in I_A$, $1=a+r\cdot x$ for some $a\in M$. Now, $y=y\cdot 1=y\cdot(a+r\cdot x)=y\cdot a+y\cdot r\cdot x$. Since, $a,x\in M$, any linear combination of $a,x$ is in $M$ and hence $y\in M$. 
\end{proof}
Now we equip the spectrum of a ring with a topology called the \textbf{Zariski topology} that makes it a compact space. If, moreover, the ring is \textbf{Boolean}, then it spectrum space is a Stone's space. 
\begin{Lemma}\label{ZarBasis}
Let $R$ be a commutative ring with identity $1$. For $x\in R$, set $$V_x=\{P\in Spec(R):x\notin P\}.$$ 
($a$) For all $x,y\in R$, $V_{x\cdot y}=V_x\cap V_y$.\\
($b$) The family $\{V_x:x\in R\}$ is a basis for a topology $\mathcal{T}_{Spec(R)}$ on $Spec(R)$ called the \textbf{Zariski topology}.
\end{Lemma}
\begin{proof}
($a$) Follows from the fact that for any prime ideal $P$ and $x,y\in R$, then $x\cdot y\notin P$ if and only if $x,y\notin P$.\\
($b$) Notice that an ideal $I$ is proper if and only if $1\notin I$ and hence $V_1=Spec(R)$. Now, by part ($a$), $\{V_x:x\in R\}$ is a basis. 
\end{proof}
\begin{Theorem}\label{SpecIsCompact}
Let $R$ be a commutative ring with identity $1$. The topological space $(Spec(R),\mathcal{T}_{Spec(R)})$ is compact.
\end{Theorem}
\begin{proof}
Let $\{O_\alpha:\alpha\in J\}$ be an open cover of $Spec(R)$. Since $\{V_x:x\in R\}$ is a basis for $Spec(R)$, there exists $A\subset R$ such that $\bigcup\limits_{\alpha\in J}O_\alpha=\bigcup\limits_{x\in A}V_x$. Let $I_A$ be the ideal generated by $A$. We claim that $I_A=R$. Seeking contradiction, assume that $I_A$ is a proper ideal. By \ref{IdealsProperties}, we can choose a maximal ideal $M$ containing $I_A$. Since $M$ is prime, $M\in Spec(R)=\bigcup\limits_{x\in A}V_x$, so $M\in V_x$ for some $x\in A$ which contradicts the fact that $A\subset I_A\subset M$. Now, since $1\in R=I_A$, $1=r_{1}\cdot a_{1}+...+r_{n}\cdot a_{n}$ for some $n\in\mathbb{N}$, $a_i\in A$ and $r_i\in R$. We claim that $Spec(R)=\bigcup\limits_{i=1}^n V_{a_i}$. Indeed, for $P\in Spec(R)$, since $1=r_{1}\cdot a_{1}+...+r_{n}\cdot a_{n}\notin P$, there exists $1\leq i\leq n$ such that $r_{i}\cdot a_{i}\notin P$ and therefore $P\in V_{a_i}$. 
\end{proof}
\begin{Definition}
Let $R$ be a ring with identity $1$ and $x\in R$:\\
$\bullet$ $x$ is called an \textbf{idempotent} element if $x^2=x$. \\
$\bullet$ $R$ is called a \textbf{Boolean} ring if every element of $R$ is an idempotent element. 
\end{Definition}
One can see that an element $x$ of a ring $R$ with identity $1$ is idempotent if and only if $x\cdot(1-x)=0$. In particular, for any prime ideal $P$ and idempotent element $x$, then as $x\cdot(1-x)=0\in P$, either $x\in P$ or $1-x\in P$, and not both (as if $x, 1-x\in P$ then $1=x+1-x\in P$). Moreover, a Boolean ring is automatically a commutative with characteristic 2, that is $2x=x+x=0$ and hence $x=-x$ for all $x\in R$. This follows from the observation:\\
$\bullet$ $2x=4x-2x=4x^2-2x=(2x)^2-2x=0$.\\
$\bullet$ $x\cdot y=x\cdot y+2(x+y)=x\cdot y+x+y+(x+y)^2=x\cdot y+x+y+x^2+y^2+x\cdot y+y\cdot x\\=2(x\cdot y+x+y)+y\cdot x=y\cdot x$.
\begin{Lemma}\label{ClopenBasis}
Let $R$ be a commutative ring with identity $1$, and $x\in R$ be idempotent element; then $V_x$ is a clopen subset of $(Spec(R),\mathcal{T}_{Spec(R)})$.
\end{Lemma}
\begin{proof}
Notice $Spec(R)\setminus V_x=V_{1-x}$.
\end{proof}
\begin{Theorem}
Let $R$ be a Boolean ring; then $(Spec(R),\mathcal{T}_{Spec(R)})$ is a Stone's space.
\end{Theorem}
\begin{proof}
By \ref{SpecIsCompact}, $(Spec(R),\mathcal{T}_{Spec(R)})$ is compact, and by \ref{ClopenBasis}, $(Spec(R),\mathcal{T}_{Spec(R)})$ admits a basis consists of clopen subsets. Hence, it suffices to show that $(Spec(R),\mathcal{T}_{Spec(R)})$ is Hausdorff to conclude that $(Spec(R),\mathcal{T}_{Spec(R)})$ is a Stone's space. Let $P_1,P_2\in Spec(R)$ be distinct prime ideals and let $x\in P_1\setminus P_2$. Notice that $V_{1-x}$ and $V_{x}$ are open disjoint neighborhoods of $P_1$ and $P_2$, respectively.
\end{proof}
The relation between Boolean algebras and Boolean rings is demonstrated in following proposition. 
\begin{Proposition}
($a$) Every Boolean algebra induces a Boolean ring.\\
($b$) Every Boolean ring induces a Boolean algebra.
\end{Proposition}
\begin{proof}
($a$) Let $(B,+_{B},\cdot,\neg)$ be Boolean algebra; the $3$-tuple $(R_{B},+_{R_{B}},\cdot)$ is a Boolean ring, where $R_{B}:=B$, and for all $x,y\in R_{B}$, $x+_{R_{B}}y:=x\cdot (\neg y)+_{B}y\cdot (\neg x)$. $(R_{B},+_{R_{B}},\cdot)$ is called the Boolean ring \textbf{induced} from the Boolean algebra $(B,+_{B},\cdot,\neg )$.\\
($b$) Let $(R,+_{R},\cdot)$ be a Boolean ring; the $4$-tuple $(B_{R},+_{B_{R}},\cdot,\neg)$ is a Boolean algebra, where $B_{R}:=R$, and for all $x,y\in B_{R}$, $x+_{B_{R}}y:=x+_{R}y+_{R}x\cdot y$ and $\neg x:=1+_{R}x$. $(B_{R},+_{B_{R}},\cdot,\neg)$ is called the Boolean algebra \textbf{induced} from the Boolean ring $(R,+_{R},\cdot)$.
\end{proof}
Ideals (prime ideals) of a Boolean algebra are ideals (prime ideals) of its induced Boolean ring.
\begin{Lemma}\label{RelationBtwIdeals}
Let $(B,+_{B},\cdot,\neg)$ be Boolean algebra and $(R_{B},+_{R_B},\cdot)$ be induced Boolean ring:\\
($a$) If $I$ is an ideal of $B$, then $I$ is an ideal of $R_B$.\\
($a$) If $I$ is a prime ideal of $B$, then $I$ is a prime  ideal of $R_B$.
\end{Lemma}
\begin{proof}
Assume that $I$ is an ideal (prime) of $B$. Since $(B,+_{B},\cdot,\neg)$ and $(R_{B},+_{R_B},\cdot)$ are equipped with the same multiplication operation, it suffices to show that if $x,y\in I$, then  $x+_{R} y\in I$. Notice that $x\cdot (\neg y), y\cdot (\neg x)\in I$ and hence $x\cdot (\neg y)+_{B} y\cdot (\neg x)= x+_{R_B} y\in I$.
\end{proof}
\begin{Lemma}\label{EveryClopenIsAbasis}
Let $R$ be a Boolean ring;\\
($a$) For all $x,y\in R$, $V_{x+y+x\cdot y}=V_x\cup V_y$.\\
($b$) If $V$ is a clopen subset of $Spec(R)$, then there exists $x\in R$ such that $V=V_x$. 
\end{Lemma}
\begin{proof}
($a$) Let $P\in V_{x+y+x\cdot y}$; if $x,y\in P$, then $x+y+x\cdot y\in P$, a contradiction. Without loss of generality, we may assume that $x\notin P$ and hence $P\in V_x$ which implies that $V_{x+y+x\cdot y}\subset V_x\cup V_y$. Conversely, assume that $P\in V_x\cup V_y$. Without loss of generality, we may assume that $P\in V_x$ and hence $x\notin P$. Since $2x=0$, $1-x=1+x\in P$ and hence $(1+x)\cdot (1+y)=1+(x+y+x\cdot y)\in P$. Since $(x+y+x\cdot y)$ is idempotent,  $(x+y+x\cdot y)\notin P$, and therefore $P\in V_{x+y+x\cdot y}$, as wanted.\\
($b$) Assume that $V$ is a clopen subset of $Spec(R)$. On one hand, since $V$ is open there exists a subset $A\subset R$ such that $V:=\bigcup\limits_{a\in A} V_{a}$. On the other hand, since $V$ is a closed subset of a compact space, it is compact and hence there exist $a_1,...,a_n\in A$ such that $V=\bigcup\limits_{i=1}^n V_{a_i}$. By part ($a$), there exists $x\in R$ such that $V=\bigcup\limits_{i=1}^n V_{a_i}=V_x$.
\end{proof}
\begin{Theorem}\label{StoneAsSpec}
 Let $(B,+_{{\textstyle\mathstrut}B},\cdot,\neg)$ be Boolean algebra and $(R_{B},+_{R_{B}},\cdot)$ be the Boolean ring induced from $(B,+_{B},\cdot,\neg)$. Then, $Spec(R_{B})$ is homeomorphic to $\mathcal{U}(B)$. 
\end{Theorem}
\begin{proof}
By \ref{StoneRepThOfBooleanAlg} and \ref{StoneRepThOfStoneSpace}, it suffices to show that the Boolean algebras $Clop(Spec(R_{B}))$ and $Clop(\mathcal{U}(B))$ are isomorphic. To that end, define  $\varphi:Clop(\mathcal{U}(B))\to Clop(Spec(R_{B}))$ that maps $U_x$ to $V_x$ for all $x\in B$.\\
$\bullet$ $\varphi$ is homomorphism. Indeed, let $x,y\in B$, then:\\
- $\varphi(U_x\cup U_y)=\varphi(U_{x+_{{\textstyle\mathstrut}B} y})=V_{x+_{{\textstyle\mathstrut}B} y}=V_{x+_{R_{B}}y+_{R_{B}}x\cdot y}=V_x\cup V_y$, by \ref{EveryClopenIsAbasis} part ($a$).\\
- $\varphi(U_x\cap U_y)=\varphi(U_{x\cdot y})=V_{x\cdot y}=V_x\cap V_y$, by \ref{PropertiesOfBasisElements} and \ref{ZarBasis}.\\
- Since $U_0=V_0=\emptyset$ and $U_1=V_1=B$, $\varphi$ is homomorphism.\\
$\bullet$ $\varphi$ is injective. Indeed, let  $U_x,U_y\in Clop(\mathcal{U}(B))$ be distinct and let $F\in U_x\setminus U_y$. Let $P:=B\setminus F$, then by \ref{RelationBtwIdeals}, $P\in Spec(R_B)$. Moreover, since $x\in F$ and $y\notin F$, $x\notin P$ and $y\in P$ and hence $P\in V_x\setminus V_y$.\\
$\bullet$ The surjectivity of $\varphi$ follows from \ref{EveryClopenIsAbasis} part ($b$). Therefore, the homomorphism $\varphi:Clop(\mathcal{U}(B))\to Clop(Spec(R_{B}))$ is an isomorphism as desired. 
\end{proof}
\section{Compactifications of a discrete space}
Let $X$ and $Y$ be topological spaces. $Y$ is called a \textbf{compactification} of $X$ if $Y$ is a compact Hausdorff space that contains a dense subset homeomorphic to $X$. In the case of discrete spaces, every compactification of a discrete space $X$ is determined by a sub-Boolean algebra $\mathcal{B}\subset \mathcal{P}(X)$ that separates points of $X$; and conversely, every sub-Boolean algebra $\mathcal{B}\subset \mathcal{P}(X)$ that separates points of $X$ determines a compactification of $X$. Every compactification of $X$ is totally disconnected. The domination relation between any two compactifications of $X$ is determined by the relation between their corresponding sub-Boolean algebras.
\begin{Lemma}\label{DenseSupOfStone}
Let $X$ be a non-empty set and $\mathcal{B}\subset \mathcal{P}(X)$ be a sub-Boolean algebra. The family $\{F_x:x\in X\}$, where $F_x:=\{A\in \mathcal{B}: x\in A\}$, is a dense discrete subspace of $\mathcal{U}(\mathcal{B})$.
\end{Lemma}
\begin{proof}
Notice that a basis element of $\mathcal{U}(\mathcal{B})$ is of the form $U_{A}=\{F\in \mathcal{U}(\mathcal{B}): A\in F\}$, and that for $x\in X$, $F_x:=\{A\in \mathcal{B}: x\in A\}$ is an ultrafilter of $\mathcal{B}$. Now, for any $a\in A$, $F_a\in U_{A}\cap \{F_x:x\in X\}$ and hence $\{F_x:x\in X\}$ is a dense subset of $\mathcal{U}(\mathcal{B})$. Moreover, for any $x\in X$, the singleton $\{F_x\}$ is an open set as $\{F_x\}=U_{\{x\}}$ is a singleton. Consequently, $\{F_x:x\in X\}$ is a discrete subspace of $\mathcal{U}(\mathcal{B})$.
\end{proof}
\begin{Definition}
Let $X$ be a non-empty set and $\mathcal{B}\subset \mathcal{P}(X)$ be a subset. $\mathcal{B}$ is said to \textbf{separate points} of $X$ if for all $x,y\in X$ such that $x\neq y$, there exists $A\in \mathcal{B}$ such that $x\in A$ and $y\in X\setminus A$. 
\end{Definition}
\begin{Theorem}\label{BooleanThatSepPointsDetrComption}
Let $X$ be a non-empty set and $\mathcal{B}\subset \mathcal{P}(X)$ be a sub-Boolean algebra. If $\mathcal{B}$ separates points of $X$, then $\mathcal{U}(\mathcal{B})$ is a compactification of $X$. 
\end{Theorem}
\begin{proof}
Define $\alpha:X\to \mathcal{U}(\mathcal{B})$ that sends each $x\in X$ to $F_x\in \mathcal{U}(\mathcal{B})$.\\
$\bullet$ $\alpha:X\to \mathcal{U}(\mathcal{B})$ is continuous as $X$ is discrete.\\
$\bullet$ It is injective since $\mathcal{B}$ separates points of $X$. Indeed, if $x,y\in X$ are distinct, then there exists $A\in \mathcal{B}$ such that $x\in A$ and $y\in X\setminus A$. Because $A\in F_x$ and $X\setminus A\in F_y$, $\alpha(x)\neq\alpha(y)$.\\
$\bullet$ $\alpha:X\to \mathcal{U}(\mathcal{B})$ is open as for all $x\in X$, $\alpha(\{x\})=\{F_x\}=U_{\{x\}}$ by \ref{DenseSupOfStone}.\\
$\bullet$ $\alpha(X)=\{F_x:x\in X\}$ which is a dense subset of $\mathcal{U}(\mathcal{B})$. Therefore, $\mathcal{U}(\mathcal{B})$ is a compactification of $X$. 
\end{proof}
\begin{Example}
Let $G$ be a group acts on $X$; then $Comm_G(X)$ is a sub-Boolean algebra of $\mathcal{P}(X)$ that separates points of $X$. The compactification $\mathcal{U}(Comm_G(X))$ of $X$ is called the \textbf{end compactification}.
\end{Example}
\begin{Example}
Let $X$ be a non-empty set, the family $\mathcal{B}_0:=\{A\in \mathcal{P}(X):$ $A$ is finite or $X\setminus A$ is finite$\}$ is a sub-Boolean algebra of $\mathcal{P}(X)$ that separates points of $X$. We will prove that the compactification $\mathcal{U}(\mathcal{B}_0)$ is the one-point compactification of $X$ \cite{Corn}.
\end{Example}
\subsection{The one-point compactification of a discrete space}
Let $Y:=X\cup\{\infty\}$ be the one-point compactification of the discrete space $X$. It can be easily seen that a subset $A\subset Y$ is clopen if and only if either $A$ is a finite subset of $X$ or $Y\setminus A$ is finite. We aim to show that the one-point compactification of $X$ is the Stone's space of the Boolean algebra $\mathcal{B}_0:=\{A\in \mathcal{P}(X):$ $A$ is finite or $X\setminus A$ is finite$\}$. 
\begin{Theorem}
Given a non-empty set $X$, the family $\mathcal{B}_0:=\{A\in \mathcal{P}(X):$ $A$ is finite or $X\setminus A$ is finite$\}$ is a sub-Boolean algebra of $\mathcal{P}(X)$ that separates points of $X$. Moreover, as a discrete space $\mathcal{U}(\mathcal{B}_0)$ is the one-point compactification of $X$. 
\end{Theorem}
\begin{proof}
The first part is straightforward. Let $Y:=X\cup\{\infty\}$ be the one-point compactification of $X$. As noticed a subset $A\subset Y$ is clopen if and only if either $A$ is a finite subset of $X$ or $Y\setminus A$ is finite. In particular, if $A\in Clop(Y)$, then $A\setminus\{\infty\}\in \mathcal{B}_0$. Define the map $\varphi:Clop(Y)\to \mathcal{B}_0$ that sends $A\in Clop(Y)$ to $A\setminus\{\infty\}$. It is straightforward to see that $\varphi:Clop(Y)\to \mathcal{B}_0$ is an isomorphism and $Y$ is a Stone's space. Hence, by \ref{IsoAndHomeo}, $Y=\mathcal{U}(Clop(Y))$ is homeomorphic to the Stone's space $\mathcal{U}(\mathcal{B}_0)$.
\end{proof}
\subsection{The Stone-$\breve{C}$ech compactification of a discrete space}
Given a compactification $Y$ of $X$ with an embedding $\alpha:X\to Y$ such that $\alpha(X)$ is dense in $Y$. The compactification $Y$ is called the \textbf{Stone-$\breve{C}$ech compactification} if for any compact Hausdorff space $K$, every continuous map $f:\alpha(X)\to K$ admits a continuous extension $\widetilde{f}:Y\to K$. For the discrete space $X$, the topological space $\mathcal{U}(\mathcal{P}(X))$ is a compactification of $X$; we aim to show that $\mathcal{U}(\mathcal{P}(X))$ is the Stone-$\breve{C}$ech compactification of $X$. We first give a characterization for the continuity of a map between topological spaces in terms of convergence of filters.

\begin{Definition}
Let $X$ be a topological space, and $F$ be a filter of $\mathcal{P}(X)$. We say that $F$ \textbf{converges} to $x\in X$ if for every neighborhood $U$ of $X$, there exists $A\in F$ such that $A\subset U$. 
\end{Definition}
\begin{Example}
Let $X$ be a topological space, and $x\in X$. The family $\mathcal{N}(x)$ of all neighborhoods of $x$ is a filter called the \textbf{neighborhood filter} of $x$. Clearly, $\mathcal{N}(x)$ converges to $x$. Moreover, a filter $F$ converges to $x$ if and only if $\mathcal{N}(x)\subset F$.
\end{Example}
A function $f:X\to Y$ between sets maps filters (ultrafilters) of $\mathcal{P}(X)$ to filters (ultrafilters) of $\mathcal{P}(Y)$.
\begin{Lemma}
Let $f:X\to Y$ be a function. Then:\\
($a$) If $F$ is a filter of $\mathcal{P}(X)$, then $F_f:=\{f(A):A\in F\}$ is a filter of $\mathcal{P}(Y)$.\\
($b$) If $F$ is an ultrafilter of $\mathcal{P}(X)$, then $F_f:=\{f(A):A\in F\}$ is an ultrafilter of $\mathcal{P}(Y)$.
\end{Lemma}
\begin{proof}
($a$) Let $F$ be a filter of $\mathcal{P}(X)$; observe that for any $B\in \mathcal{P}(Y)$, $f(f^{-1}(B))=B$. Now:\\
$\bullet$ Let $B_1,B_2\in F_f$, then there exist $A_1,A_2\in F$ such that $f(A_1)=B_1$ and $f(A_2)=B_2$. Since $A_1\cap A_2\in F$ and $A_1\cap A_2\subset f^{-1}(B_1\cap B_2)$ and hence $f^{-1}(B_1\cap B_2)\in F$. Therefore, $f(f^{-1}(B_1\cap B_2))=B_1\cap B_2\in F_f$.\\
$\bullet$ $B\in F_f$ and $C\subset Y$ such that $B\subset C$. Let $A\in F$ such that $f(A)=B$; since $A\subset f^{-1}(C)$, then $f^{-1}(C)\in F$ and hence $f(f^{-1}(C))=C\in F_f$.\\
($b$) Let $F$ be an ultrafilter of $\mathcal{P}(X)$, then $F_f$ is a filter of $\mathcal{P}(Y)$. Moreover, if $B\subset Y$, then either $f^{-1}(B)\in F$ or $X\setminus f^{-1}(B)=f^{-1}(Y\setminus B)\in F$ and therefore either $B=f(f^{-1}(B))\in F_f$ or $Y\setminus B=f(f^{-1}(Y\setminus B))\in F_f$.
\end{proof}
\begin{Lemma}\label{CharOfContByFiltersConv}
A map $f:X\to Y$ between topological spaces is continuous if and only if for every $x\in X$ and every filter $F$ of $\mathcal{P}(X)$ that converges to $x$, the filter $F_f$ converges to $f(x)$.
\end{Lemma}
\begin{proof}
Assume that $f:X\to Y$ is continuous, and $F$ is a filter of $\mathcal{P}(X)$ that converges to $x$. Let $V$ be a neighborhood of $f(x)$, then $f^{-1}(V)$ is a neighborhood of $x$ and consequently $f^{-1}(V)\in F$. Hence, $V=f(f^{-1}(V))\in F_f$ which implies that $F_f$ converges to $f(x)$.

Conversely, assume that for every $x\in X$ and every filter $F$ of $\mathcal{P}(X)$ that converges to $x$, the filter $F_f$ converges to $f(x)$. Since the neighborhood filter $F=\mathcal{N}(x)$ of $x$ converges to $x$, the filter $F_f$ converges to $f(x)$. Let $V\subset Y$ be a neighborhood of $f(x)$, then $V\in F_f$ and therefore there exists $U\in F$ such that $f(U)=V$ which proves that $f:X\to Y$ is continuous.
\end{proof}
\begin{Theorem}\label{StoneChecCompOfDiscSpace}
Let $X$ be a non-empty set equipped with the discrete topology. The compactification $\mathcal{U}(\mathcal{P}(X))$ of $X$ is the Stone-$\breve{C}$ech compactification.
\end{Theorem}
\begin{proof}
By \ref{BooleanThatSepPointsDetrComption}, $\mathcal{U}(\mathcal{P}(X))$ is a compactification of $X$. Let $K$ be a compact Hausdorff space, and $f:\alpha(X)\to K$ be a continuous map, where $\alpha:X\to \mathcal{U}(\mathcal{P}(X))$ is the embedding defined in the proof of \ref{BooleanThatSepPointsDetrComption}.\\
\textbf{Claim 1:} For every $F\in \mathcal{U}(\mathcal{P}(X))$, $\bigcap\limits_{A\in F}cl(f( \alpha(A)))$ is a singleton.\\
\textbf{Proof of claim 1:} Put $g:=f\circ \alpha$. Notice that the family $F$ satisfies the finite intersection property and hence the family $\{g(A):A\in F\}$ satisfies the finite intersection property. Therefore, the family of the closed sets $\{cl(g(A)):A\in F\}$ satisfies the finite intersection property. Consequently, $\bigcap\limits_{A\in F}cl(g(A))\neq \emptyset$, as $K$ is compact. Since $F$ is an ultrafilter, for every $O\subset X$, either $O\in F$ or $X\setminus O\in F$; in particular, if $O_1,O_2\subset X$ are disjoint subsets of $X$, then there exists $A\in F$ such that either $A\cap O_1=\emptyset$ and  $A\cap O_2\neq\emptyset$ or $A\cap O_1\neq\emptyset$ and  $A\cap O_2=\emptyset$. Seeking contradiction assume that $a,b\in \bigcap\limits_{A\in F}cl(g(A))$ are distinct, and choose two disjoint neighborhoods $U_a$ and $U_b$ of $a$ and $b$, respectively. Since $g^{-1}(U_a)$ and $g^{-1}(U_b)$ are disjoint, we may choose, without loss of generality, $A\in F$ such that $A\cap g^{-1}(U_a)\neq \emptyset$ and $A\cap g^{-1}(U_b)=\emptyset$. Hence, $g(A)\cap U_a\neq\emptyset$ and $g(A)\cap U_b=\emptyset$, and thus $b\notin cl(g(A))$, a contradiction. Therefore, there exists $a_F\in K$ such that $\bigcap\limits_{A\in F}cl(g(A))=\{a_{{\textstyle\mathstrut}F}\}$.\\
\textbf{Claim 2:} Define $\widetilde{f}:\mathcal{U}(\mathcal{P}(X))\to K$ by $\widetilde{f}(F)=a_{{\textstyle\mathstrut}F}$. Then $\widetilde{f}$ is an extension of $f$.\\
\textbf{Proof of claim 2:} For $x\in X$, then $\widetilde{f}(F_x)=a_{{\textstyle\mathstrut}F_x}$. On the other hand, $F_x\in \alpha(A)$ for all $A\in F_x$ and hence $f(F_x)\in \bigcap\limits_{A\in F}f( \alpha(A)))\subset \bigcap\limits_{A\in F}cl(f( \alpha(A)))=\{a_{{\textstyle\mathstrut}F_x}\}$.\\
\textbf{Claim 3:} $cl(\widetilde{f}(U_A))=cl(g(A))$ for all $A\in \mathcal{P}(X)$; where: $$U_A=\{F\in \mathcal{U}(\mathcal{P}(X)):A\in F\}.$$\\
\textbf{Proof of claim 3:} Let $a\in \widetilde{f}(U_A)$, then there exists $F\in U_A$ such that $a=a_{{\textstyle\mathstrut}F}$ where $\{a_{{\textstyle\mathstrut}F}\}=\bigcap\limits_{A\in F}cl(g(A))\subset cl(g(A))$, and hence $cl(\widetilde{f}(U_A))\subset cl(g(A))$. On the other hand, since $F_a\in U_A$ for all $a\in A$, thus $g(a)=f(F_a)=\widetilde{f}(F_a)\in \widetilde{f}(U_A)$. Therefore, $cl(g(A))\subset cl(\widetilde{f}(U_A))$.\\
\textbf{Claim 4:} $\widetilde{f}:\mathcal{U}(\mathcal{P}(X))\to K$ is continuous. \\
\textbf{Proof of claim 4:} By \ref{CharOfContByFiltersConv}, it suffices to show that if $\mathcal{F}$ is a filter of $\mathcal{P}(\mathcal{U}(\mathcal{P}(X)))$ that converges to $F\in \mathcal{U}(\mathcal{P}(X))$, then $\mathcal{F}_{\widetilde{f}}$ converges to $\widetilde{f}(F)$. Notice that for all $A\in F$, $F\in U_A$, and since $\mathcal{F}$ converges to $F$ and $ U_A$ is a neighborhood of $F$, $ U_A$ must be in $\mathcal{F}$. Consequently, $\{\widetilde{f}(U_A):A\in F\}\subset \mathcal{F}_{\widetilde{f}}$. Since $\widetilde{f}(U_A)\subset cl(\widetilde{f}(U_A))=cl(g(A))$ and $\mathcal{F}_{\widetilde{f}}$ is a filter, so $\{cl(g(A)):A\in F\}\subset \mathcal{F}_{\widetilde{f}}$. Now, let $O$ be a neighborhood of $\widetilde{f}(F)=a_{{\textstyle\mathstrut}F}$. Choose a relatively compact neighborhood $V$ of $a_{{\textstyle\mathstrut}F}\in \bigcap\limits_{A\in F}cl(g(A))$ such that $V\subset cl(V)\subset O$. As $V\cap g(A)\neq\emptyset$ for all $A\in F$ and that $F_g$ is an ultrafilter, hence, by \ref{UltraAreMaxSatTheFPP}, $V\in F_g$ and therefore there exists $B\in F$ such that $V=g(B)$. Thus, $cl(V)\in \mathcal{F}_{\widetilde{f}}$ and $cl(V)\subset O$ which implies that $\mathcal{F}_{\widetilde{f}}$ converges to $\widetilde{f}(F)$. As a result, $\mathcal{U}(\mathcal{P}(X))$ is the Stone-$\breve{C}$ech compactification of $X$.
\end{proof}
\begin{Corollary}\label{WaysToConsSTCompOfDiscrete}
Let $X$ be a non-empty set equipped with the discrete topology. The Stone-$\breve{C}$ech compactification $\beta X$ of $X$ is totally disconnected. Moreover, $\beta X$ is homeomorphic to each of the following Stone's spaces:\\
($1$) $(\mathcal{U}(\mathcal{P}(X)),\mathcal{T}_{\mathcal{U}(\mathcal{P}(X))})$.\\
($2$) $Hom(\mathcal{P}(X),\mathbb{Z}_2)$.\\
($3$) $Spec(R_{\mathcal{P}(X)})$, such that $(R_{\mathcal{P}(X)},+,\cdot)$ is the ring  where $R_{\mathcal{P}(X)}=\mathcal{P}(X)$ and for all $A,B\in \mathcal{P}(X)$:\\
$\bullet$ $A+B:=A\Delta B=(A\setminus B)\cup (B\setminus A)$;\\
$\bullet$ $A\cdot B:=A\cap B$.
\end{Corollary}
\begin{proof}
Apply \ref{StoneChecCompOfDiscSpace}, \ref{StoneAsSetOfHom} and \ref{StoneAsSpec}.
\end{proof}
\begin{Corollary}
Let $X$ be a non-empty set equipped with the discrete topology, and $Y$ be a compactification of $X$, then $Y$ is totally disconnected. 
\end{Corollary}
\begin{proof}
Let $\beta X$ be the Stone-$\breve{C}$ech compactification of $X$, then $Y$ is a quotient space of the totally disconnected space $\beta X$, and therefore $Y$ is totally disconnected.
\end{proof}
When $X$ is infinite, its Stone-$\breve{C}$ech compactification is non-metrizable. 
\begin{Theorem}
Let $X$ be an infinite discrete space. The Stone-$\breve{C}$ech compactification $\beta X$ is non-metrizable.
\end{Theorem}
\begin{proof}
Seeking contradiction, assume that $\beta X$ is metrizable. Choose $\omega\in \beta X\setminus X$ and a sequence $(x_n)_{n\in \mathbb{N}}$ of distinct elements of $X$ that converges to $\omega$. Let $A:=\{x_{2n}:n\in\mathbb{N}\}$, and $\chi_A:X\to \{0,1\}$ be the characteristic function of $A$. Since $\chi_A$ is continuous, it admits a continuous extension $\widetilde{\chi_A}:\beta X\to \{0,1\}$. On one hand, we have $\widetilde{\chi_A}(cl_{\beta X}A)=\{1\}$, and hence $\widetilde{\chi_A}(\omega)=1$. On the other hand, we have $\widetilde{\chi_A}(x_{2n-1})=0$ for all $n\in\mathbb{N}$ and $(x_{2n-1})_{n\in\mathbb{N}}$ converges to $\omega$, a contradiction.
\end{proof}
Now we prove the converse of \ref{BooleanThatSepPointsDetrComption}.
\begin{Theorem}\label{ForAComptionThereExistsAbooleanDetIt}
Let $Y$ be a compactification of the discrete space $X$; there exists a sub-Boolean algebra $\mathcal{B}\subset \mathcal{P}(X)$ that separates points of $X$ such that $Y$ is homeomorphic to $\mathcal{U}(\mathcal{B})$.
\end{Theorem} 
\begin{proof}
Without loss of generality we may assume that $X\subset Y$. Put $\mathcal{B}:=\{A\cap X: A\in Clop(Y)\}$. It is straightforward to see that $\mathcal{B}\subset \mathcal{P}(X)$ is a sub-Boolean algebra that separates points of $X$. Therefore, $\mathcal{U}(\mathcal{B})$ is a compactification of $X$. Consider the a map $\varphi:Clop(Y)\to \mathcal{B}$ that sends every $A\in Clop(Y)$ to $A\cap X$. It is evident to see that $\varphi:Clop(Y)\to \mathcal{B}$ is a surjective homomorphism. Furthermore, let $A,B\in Clop(Y)$ be distinct. Without loss of generality assume that $A\setminus B$ is non-empty. Since $A\setminus B$ is an open subset of $Y$ and hence $(A\setminus B)\cap X\neq \emptyset$. Thus, $A\cap X\neq B\cap X$, which shows that $\varphi:Clop(Y)\to \mathcal{B}$ is injective. By \ref{IsoAndHomeo}, $Y$ is homeomorphic to $\mathcal{U}(\mathcal{B})$.
\end{proof}
\begin{Corollary}\label{ComptionOFDiscreteSp}
Let $X$ be a non-empty set equipped with the discrete topology. A topological space $Y$ is a compactification of $X$ if and only if there exists a sub-Boolean algebra $\mathcal{B}\subset \mathcal{P}(X)$ that separates points of $X$ such that $Y$ is homeomorphic to $\mathcal{U}(\mathcal{B})$.
\end{Corollary}
\begin{proof}
Apply \ref{BooleanThatSepPointsDetrComption} and \ref{ForAComptionThereExistsAbooleanDetIt}.
\end{proof}
Recall if $Y$ and $Z$ are compactifications of $X$, we say that $Y$ \textbf{dominates} $Z$, or $Z$ is \textbf{dominated by} $Y$, if there exists a continuous function $f:Y\to Z$ such that $f|_{X}=id_{X}$. The continuous function $f:Y\to Z$ is called a \textbf{domination}.
\begin{Theorem}
Let $Y:=\mathcal{U}(\mathcal{B}_Y)$ and $Z:=\mathcal{U}(\mathcal{B}_Z)$ be compactifications of the discrete space $X$, where $\mathcal{B}_Y$ and $\mathcal{B}_Z$ sub-Boolean algebras of $\mathcal{P}(X)$ that separate points of $X$. $Y$ dominates $Z$ if and only if $\mathcal{B}_Z$ is isomorphic to a sub-Boolean algebra of $\mathcal{B}_Y$. 
\end{Theorem}
\begin{proof}
By \ref{ClopBoolean}, a continuous function $f:\mathcal{U}(\mathcal{B}_Y)\to \mathcal{U}(\mathcal{B}_Z)$ induces a homomorphism $\widetilde{f}:Clop(\mathcal{U}(\mathcal{B}_Z))\to Clop(\mathcal{U}(\mathcal{B}_Y))$ given by $\widetilde{f}(A)=f^{-1}(A)$ for all $A\in Clop(\mathcal{U}(\mathcal{B}_Z))$. It can be easily seen that if   $f:\mathcal{U}(\mathcal{B}_Y)\to \mathcal{U}(\mathcal{B}_Z)$ is surjective, then $\widetilde{f}:Clop(\mathcal{U}(\mathcal{B}_Z))\to Clop(\mathcal{U}(\mathcal{B}_Y))$ is injective. 
On the other hand, we may assume that $\mathcal{B}_Z\subset \mathcal{B}_Y$. Now the inclusion $\varphi:\mathcal{B}_Z\to \mathcal{B}_Y$ is a homomorphism. By \ref{IsoBooleansHaveHomeoStonesSpaces}, it induces a continuous map $\widetilde{\varphi}:\mathcal{U}(\mathcal{B}_Y)\to \mathcal{U}(\mathcal{B}_Z)$ that sends $F\in \mathcal{U}(\mathcal{B}_Y)$ to $\varphi^{-1}(F)=F\cap \mathcal{B}_Z$. Notice that every ultrafilter $F_Z$ of $\mathcal{B}_Z$ is contained in an ultrafilter $F_Y$ of $\mathcal{B}_Y$ with $F_Z=F_Y\cap \mathcal{B}_Z$ which shows that $\widetilde{\varphi}:\mathcal{U}(\mathcal{B}_Y)\to \mathcal{U}(\mathcal{B}_Z)$ is surjective. Moreover, since $X$ is homeomorphic to $\{F^{Y}_x:x\in X\}$ and homeomorphic to $\{F^{Z}_x:x\in X\}$ where $F^{Y}_x=\{A\in \mathcal{B}_Y:x\in A\}$ and $F^{Z}_x=\{A\in \mathcal{B}_Z:x\in A\}$, so $\widetilde{\varphi}:\mathcal{U}(\mathcal{B}_Y)\to \mathcal{U}(\mathcal{B}_Z)$ is a domination. 
\end{proof}

\end{document}